\documentclass{amsart}

\usepackage{layout} 
\usepackage[top=3cm, bottom=3cm, left=2cm, right=2cm]{geometry} 
\usepackage[english]{babel}
\usepackage[utf8]{inputenc}
\usepackage[T1]{fontenc}
\usepackage{amsmath}
\usepackage{amssymb}
\usepackage{mathrsfs}
\usepackage{amsthm}
\usepackage{enumerate} 
\usepackage{enumitem}
\usepackage{comment}
\usepackage{mathtools}
\usepackage{esint}
\usepackage{faktor}
\usepackage{color}
\usepackage[capitalize]{cleveref}

\newtheorem{theorem}{Theorem}[section]
\newtheorem{proposition}[theorem]{Proposition}
\newtheorem{corollary}[theorem]{Corollary}
\newtheorem{lemma}[theorem]{Lemma}
\newtheorem{remark}[theorem]{Remark}

\newcommand{\scal}[2]{\left\langle #1,#2 \right\rangle}
\newcommand{\g}{\nabla}

\newcommand{\di}{\mathrm{div}}
\newcommand{\lap}{\Delta}
\newcommand{\dr}{\partial}
\newcommand{\tr}{\mathrm{tr}}
\newcommand{\vol}{\mathrm{vol}}

\newcommand{\R}{\mathbb{R}}

\newcommand{\N}{\mathbb{N}}

\newcommand{\s}{\mathbb{S}}

\newcommand{\B}{\mathbb{B}}

\newcommand{\inter}[2]{[\![#1,#2]\!]}
\newcommand{\abs}[1]{ \left| #1 \right| }

\newcommand{\eps}{\varepsilon}
\newcommand{\ve}{\varepsilon}
\newcommand{\vp}{\varphi}
\newcommand{\id}{\mathrm{id}}

\newcommand{\loc}{\mathrm{loc}}

\newcommand{\fl}{\mathrm{flat}}
\newcommand{\cst}{\mathrm{cst}}
\newcommand{\aleq}{\lesssim}

\DeclareMathOperator*{\esssup}{ess\,sup}
\DeclareMathOperator*{\osc}{osc}

\title[Regularity of unconstrained $p$-harmonic maps from curved domain and application]{Regularity of unconstrained $p$-harmonic maps from curved domain and application to critical $p$-Laplace systems}
\author{Dorian Martino}
\address[Dorian Martino]{ETH Zürich, Rämistrasse 101, 8092 Zürich}
\email{dorian.martino@math.ethz.ch}
\date{\today}

\begin{document}

	\begin{abstract}
		Given $p\geq 2$ and a map $g : B^n(0,1)\to S_n^{++}$, where $S_n^{++}$ is the group of positively definite matrices, we study critical points of the following functional:
		\begin{align*}
			v\in W^{1,p}\left(B^n(0,1);\R^N \right) \mapsto \int_{B^n(0,1)} |\g v|^p_g\, d\vol_g = \int_{B^n(0,1)} \left( g^{\alpha\beta}(x) \scal{\dr_\alpha v(x)}{\dr_\beta  v(x)} \right)^{\frac{p}{2}}\, \sqrt{\det g(x)}\, dx.
		\end{align*}
		We show that if $g$ is uniformly close to a constant matrix, then $v$ is locally Hölder-continuous. If $g$ is Hölder-continuous, we show that $\g v$ is locally Hölder-continuous. As an application, we prove that any Hölder-continuous solution to $|\lap_{g,p}u|\aleq |\g u|^p_g$ satisfies additional regularity properties depending on the regularity of $g$. In the case $p=n$, only the continuity is assumed \textit{a priori}.
	\end{abstract}
	
	\maketitle
	
	
	\subsection*{Keywords}
	$p$-Laplacian, degenerate elliptic system, regularity.
	
	\subsection*{Mathematics Subject Classification (2010)}
	35B65, 35J92, 58J05.
	
	\section{Introduction}
	
	In order to study the regularity of free boundary harmonic maps between manifolds and their generalizations, there are two standard methods. The first one, which is useful to study minimizing or stationary maps, is by analysing blow-ups or monotonicity formulas up to the boundary, see for instance \cite{jost2019L,guo2020,moser2022} and references therein. The second one is to perform a reflection at a given point of the boundary in order to state the question as an interior regularity problem, see for instance \cite{jost2019,laurain2017,laurain2019,liu2022,mazowiecka2018,scheven2006} and references therein. To do so, the first step is to show that near the boundary, the solution has controlled oscillations. Once this step is understood and the reflection is performed, one usually obtain a system with leading-order operator having coefficients depending on the regularity of the solution itself and therefore, having very low regularity. For instance in the case of free boundary harmonic maps, one obtain a system of the form $|\lap_g u|\aleq |\g u|^2_g$, where $g$ is a metric of the form $h(\cdot,u(\cdot))$, where $h$ is regular. After showing that the oscillations of the free boundary harmonic map are small near the boundary, we deduce that $g$ is uniformly close to a constant, but \textit{a priori} not even continuous. \\
	
	In the case of $p$-harmonic maps between manifolds, only the boundary regularity of minimizers and stationary maps has been studied \cite{duzaar1998,fuchs1990,hardt1987,muller2002}. One could hope to generalize the reflection argument of Scheven \cite{scheven2006} in order to write the Neumann boundary condition as an interior problem. One would obtain a system closely related to the one of $p$-harmonic maps with degenerate metric on the domain and arbitrary target. The regularity of $p$-harmonic maps is still an open problem in full generality but it is understood for round targets, see for instance the review \cite{schikorra2017} and the more recent works \cite{miskiewicz2023,martinoschiko2024,martinoschikorra2024}. The main issue is that the regularity of unconstrained $p$-harmonic maps with a domain having a metric of low regularity is not well understood. The regularity of $p$-harmonic maps with free boundary on the round sphere have been studied by Mazowiecka\textendash Rodiac\textendash Schikorra \cite{mazowiecka2020}, based on a strategy introduced by Schikorra \cite{schikorra2018}. In order to avoid a reflection, he wrote the Neumann boundary condition as a equation along the boundary and obtain a system which is \textit{not} of the form $|\lap_p u| \aleq |\g u|^p$. He successfully obtained a new proof of the regularity up to the boundary for systems closely related to harmonic maps systems and $H$-systems in dimension 2. \\
	
	The main motivation of the present work is to generalize the reflection argument to the $p$-Laplacian setting. To do so, we study the regularity of unconstrained $p$-harmonic maps whose domain is endowed with an $L^\infty$-metric having low regularity. More precisely, let $\B$ be the $n$-dimensional unit ball and $g$ be a metric on $\B$, \textit{i.e.} a map $ g: x\in \B \mapsto (g_{\alpha\beta}(x))_{1\leq \alpha,\beta\leq n}\in S_n^{++}$ where $S_{n}^{++}$ is the set of positive definite matrices, such that $g,g^{-1}\in L^\infty(\B)$. We study the solutions $v\in W^{1,p}(\B;\R^N)$ to the following system:
	\begin{align}\label{eq:pLap_system}
		\lap_{g,p}v := \frac{1}{\sqrt{\det g} } \dr_i\left( \left( g^{\alpha\beta} \scal{\dr_\alpha v}{\dr_\beta v} \right)^\frac{p-2}{2} g^{ij} \, \sqrt{\det g}\, \dr_i v \right) = 0.
	\end{align}
	Here, we used the Einstein summation convention. We also denoted $(g^{\alpha\beta})_{1\leq \alpha,\beta\leq n} = g^{-1}$ and $\scal{\cdot}{\cdot}$ the Euclidean scalar product in $\R^N$. Then, we study the improvement of regularity in the two cases where $g$ is Hölder-continuous and $C^1$.\\
	
	Solutions to \eqref{eq:pLap_system} can be obtained as local minimizers of the following functional: 
	\begin{align*}
		E_g : v\in W^{1,p}(\B;\R^N) \mapsto \int_{\B} |\g v|^p_g\, d\vol_g = \int_{\B} \left( g^{\alpha\beta}(x) \scal{\dr_\alpha v(x)}{\dr_\beta v(x)} \right)^{\frac{p}{2}}\, \sqrt{\det g(x)}\, dx.
	\end{align*}
	Minimizers of non-autonomous functionals of the form $\int F(x,Dv)dx$, with $F$ satisfying the standard growth condition $|\xi|^p \aleq |F(x,\xi)| \aleq (1+|\xi|)^p$, have been intensively studied. However, to the author's best knowledge, our specific setting is not included in  the existing literature. The case where $g$ is the flat metric has been studied by Uhlenbeck in the seminal paper \cite{uhlenbeck1977}. In particular, the case where $g$ is a constant matrix follows by a straight-forward change of variable, see \Cref{cor:g_constant}. Since then, various contributions have been done to study minimizers of functionals having non-elliptic growth, see \cite{mingione2006,mingione2021} for two surveys. If $g$ is assumed to be continuous, then solutions to \eqref{eq:pLap_system} are Hölder-continuous, see \cite[Theorem 4.1]{giaquinta1982} for the case $p<n$ and \cite[Theorem 4.3.1]{morrey2008} the case $p=n$. Here, we only assume that $g$ is $L^\infty$, with a uniform distance to a constant matrix. The regularity up to a small set have been obtained in the case where $g$ is Hölder-continuous or more regular in \cite{acerbi2005,chlebicka2021,defilippis2020M,fusco1989,giaquinta1986,goodrich2020,hamburger1992,tachikawa2024}. Everywhere regularity results have been obtained under the assumption of either $g$ being $C^2$, $v$ belonging to $L^\infty_{\loc}$ or satisfying some other specific boundary conditions \cite{bogelein2022,breit2011,carozza2011,defilippis2020,defilippis2023,diening2009,ragusa2008,ragusa2013,ragusa2016}. \\
	
	In order to study the solution $v$ to \eqref{eq:pLap_system}, we compare it on each ball $B(x,r)\subset\B^n$ to the solution $w$ having $v$ as boundary condition and satisfying \eqref{eq:pLap_system} for the constant metric $g(x)$. In order to keep track of the dependence of the constants with respect to the Hölder-norm and the $C^1$-norm of $g$, we will provide a new proof in the known cases. We emphasize that the proofs are direct in the sense that they do not involve any approximation argument by a regularized system or mollified maps. Thanks to the ideas of Kuusi--Mingione \cite{kuusi2013}, we obtain the following result, see \Cref{th:reg_pharm_curved_domain} for a precise statement.
	
	\begin{theorem}\label{thintro:unconstrained_pharm}
		Let $p\in[2,+\infty)$. Let $\lambda,\Lambda>0$ and $g : \B \to S_n^{++}$ be a metric satisfying 
		\begin{align}\label{eq:Linfty_g}
			\lambda\xi \leq g\leq \Lambda\xi.
		\end{align}
		Consider $g_0\in S_n^{++}$ a constant metric. There exists positives constants $\gamma\in(0,1)$ and $\eps_1\in(0,1)$ depending only on $p,n,N,\lambda,\Lambda$ such that the following holds. Assume that
		\begin{align}\label{hyp:g_unif_close_constant}
			\|g-g_0\|_{L^\infty(\B)} + \|\id - g^{-1}g_0 \|_{L^\infty(\B)} + \|\id - g\, g_0^{-1}\|_{L^\infty(\B)} \leq \eps_1.
		\end{align}
		Consider a map $v\in W^{1,p}(\B;\R^N)$ such that $\lap_{g,p} v= 0$ on $\B$. Then, $v$ satisfies the following properties:
		\begin{enumerate}
			\item It holds $v \in C^{0,\gamma}_{\loc}(\B)$. 
			\item If $g$ is a $C^{0,\beta}$-metric for some $\beta\in(0,1)$ and $[g]_{C^{0,\beta}(\B)}\leq \eps_1$, then $\g v\in C^{0,\gamma}_{\loc}(\B)$.
			\item If $g$ is a $C^1$-metric, then it holds $|\g v|_g^\frac{p-2}{2}\g v \in W^{1,2}_{\loc}(\B;\R^{n\times N})$ and $|\g v|^{p-2}_g\g v \in W^{1,\frac{p}{p-1}}_{\loc}(\B;\R^{n\times N})$.
		\end{enumerate}
	\end{theorem} 
	
	In the statement of the above theorem, we have denoted the Hölder seminorm by
	\begin{align*}
		[g]_{C^{0,\beta}(\B)} := \sup_{x,y\in \B} \frac{|g(x) - g(y)|}{|x-y|^{\beta}}.
	\end{align*}
	In a second part, we study the regularity of continuous solutions to systems $|\lap_{p,g} u|\aleq |\g u|^p$. We will show that such solutions are Hölder-continuous where $g$ is uniformly close to a constant matrix and $C^{1,\alpha}$ when $g$ is Hölder-continuous. This result can be seen as a generalization of \cite[Section 3]{hardt1987} or \cite[Theorem 2.3]{mou1996} to the case where the domain metric $g$ has low regularity. Furthermore, our proof is direct in the sense that it does not use an approximation argument by solutions of a regularized system. Following the same strategy as for \cref{th:reg_pharm_curved_domain}, we obtain the following result in the case $p=n$, see \cref{th:reg_nLap_critical} for a precise statement:
	
	\begin{theorem}\label{thintro:critical_system}
		Let $g$ be a $L^\infty$-metric on $\overline{\B}$ satisfying \eqref{eq:Linfty_g}-\eqref{hyp:g_unif_close_constant} and fix $\Gamma>0$. There exists $\gamma\in(0,1)$ and $\kappa>0$ depending only on $n,N,\lambda,\Lambda,\Gamma$ such that the following holds. Consider any map $u\in W^{1,n}\cap C^0(\B;\R^N)$ solution to $|\lap_{g,n}u|\leq \Gamma |\g u|^n$. Then $u$ satisfies the following properties:
		\begin{enumerate}
			\item It holds $u\in C^{0,\gamma}_{\loc}(\B;\R^N)$.
			\item If $g$ is a $C^{0,\beta}$-metric, then $u\in C^{1,\gamma}_{\loc}(\B;\R^N)$. 
		\end{enumerate}
	\end{theorem}
	
	In the general case $p\geq 2$, we assume first the Hölder-continuity, we refer also to \Cref{th:reg_nLap_critical}:
	\begin{theorem}\label{thintro:noncritical_system}
		Let $p\in[2,+\infty)$ and $\beta_1,\beta_2\in(0,1)$. Let $g$ be a $C^{0,\beta_1}$-metric on $\overline{\B}$ satisfying \eqref{eq:Linfty_g}-\eqref{hyp:g_unif_close_constant} and fix $\Gamma>0$. There exists $\alpha\in(0,1)$ and $\kappa>0$ depending only on $n,N,\lambda,\Lambda,\Gamma,p$ such that the following holds. Consider any map $u\in W^{1,p}\cap C^{0,\beta_2}(\B;\R^N)$ solution to $|\lap_{g,p}u|\leq \Gamma |\g u|^p$. Then it holds $u\in C^{1,\alpha}_{\loc}(\B;\R^N)$. 
	\end{theorem}
	
	As a corollary, we obtain the following regularity result in the case where $g$ depends on $u$.
	
	\begin{corollary}
		Let $h\in C^\infty(\B\times \R^N;S_n^{++})$ such that $((x,y)\mapsto h(x,y)^{-1})\in C^\infty(\B\times \R^N;S_n^{++})$. Let $\beta\in(0,1)$ and $p\in[2,+\infty)$. There exists $\alpha>0$ depending on $n,N,p,\|h\|_{C^1(\B\times \R^N)},\|h(\cdot)^{-1}\|_{C^1(\B \times \R^N)},\Gamma,\beta$ such that the following holds.
		Let $u\in W^{1,p}\cap C^{0,\beta}(\B;\R^N)$ be a solution to $|\lap_{g,p}u|\leq \Gamma |\g u|^p$ for the metric $g : x\mapsto h(x, u(x))$. Then it holds $u\in C^{1,\alpha}_{\loc}(\B;\R^N)$.
	\end{corollary}
	
	\subsubsection*{Organization of the paper}
	In \Cref{sec:unconstrained_pharm}, we prove \Cref{thintro:unconstrained_pharm}. In \Cref{sec:setting_unconstrained_pharm}, we discuss \Cref{th:reg_pharm_curved_domain}. In \Cref{sec:g_Linfty}, we study the case where $g$ is uniformly close to a constant matrix and we prove \Cref{item:Linfty_pharm_curved} of \Cref{th:reg_pharm_curved_domain}. In \Cref{sec:g_Holder}, we study the case where $g$ is Hölder-continuous and we prove \Cref{item:osc_pharm_curved} of \Cref{th:reg_pharm_curved_domain}. In \Cref{sec:g_C1}, we study the case where $g$ is $C^1$ and we prove \Cref{item:2nd_der_pharm_curved} of \Cref{th:reg_pharm_curved_domain}. In \Cref{sec:Critical_system}, we prove \Cref{thintro:critical_system} and \Cref{thintro:noncritical_system}.

	\subsubsection*{Acknowledgments}
	I thank Paul Laurain for his constant support and advice. This work was motivated by a project in collaboration with Katarzyna Mazowiecka and Rémy Rodiac which originated during a visit at the University of Warsaw. I would like to thank them for their invitation and their comments on previous versions of this work. I thank Tristan Rivière for stimulating discussions. 
	
	\section{Systems $\lap_{g,p}v=0$}\label{sec:unconstrained_pharm}
	
	In the rest of the paper, we will use Einstein summation convention. The ball $B(x,r)$ will always denote the Euclidean ball of center $x$ and radius $r$.\\
	
	In all this section, we denote $\B=B_1^n = B_{\R^n}(0,1)$ and we consider $g$ an $L^\infty$-metric on $\B$, \textit{i.e.} a map $g : x\in\B \mapsto (g_{\alpha\beta}(x))_{1\leq \alpha,\beta\leq n} \in S_n^{++}$ where $S_n^{++}$ is the set of positive definite matrices, and such that $g,g^{-1}\in L^\infty(\B)$. We will denote $g^{-1} = (g^{\alpha\beta})_{1\leq \alpha,\beta\leq n}$. \\
	
	Let $p\geq 2$ and $v:\B\to \R^N$ be a weak solution to 
	\begin{align}\label{eq:pharm_system_curved}
		\left\{ \begin{array}{l}
			v\in W^{1,p}(\B;\R^N),\\
			\lap_{g,p} v = \di_g\left( |\g v|^{p-2}_g \g^g v\right) = 0\ \ \ \text{in }\B.
		\end{array}
		\right.
	\end{align}
	Here, we have denoted $|\g v|_g^2 = g^{ij} \scal{\dr_i v}{\dr_j v}$, $\g^g v = (g^{ij}\dr_j v)_{1\leq i\leq n} \in \R^{n\times N}$ and for any vector field $X:\B\to \R^n$ we define $\di_g(X) = \frac{1}{\sqrt{\det g}} \dr_i(\sqrt{\det g}\, X^i)$. The weak formulation of \eqref{eq:pharm_system_curved} is the following:
	\begin{align}\label{eq:weak_formulation}
		\forall \vp \in C^\infty_c(\B;\R^N),\ \ \ & \int_{\B} |\g v(x)|^{p-2}_g g^{\alpha\beta}(x)\scal{\dr_{\alpha} v(x)}{\dr_{\beta} \vp(x)}\, \sqrt{\det g(x)}\, dx = 0.
	\end{align}
	We will denote $\xi=\delta_{ij} dx^i \otimes dx^j$ the flat metric on $\B$, where $\delta_{ij}=0$ if $i\neq j$ and $\delta_{ii}=1$.
	
	\subsection{Unconstrained $p$-harmonic maps from flat domain}
	
	In this section, we recall the regularity properties of solutions to $\lap_{\xi,p} w=0$ that we will to generalize to solutions of $\lap_{g,p}v=0$. These properties are proved in \cite{kuusi2018,uhlenbeck1977}, see also \cite[Section 2.4]{martino2024}. We then argue that, by a straight-forward change of variables, we deduce the regularity of solutions to $\lap_{g,p}v=0$ for a constant metric $g\equiv g(0)$.
	
	\begin{theorem}\label{th:reg_flat}
		There exists constants $\kappa_{\fl}>1$, $\alpha_{\fl}\in (0,1)$ and $\sigma_{\fl}\in(0,\frac{1}{2}]$ depending on $n,N,p$ such that the following holds. Let $w\in W^{1,p}(\B;\R^N)$ be a solution to $\lap_{\xi,p}w=0$. Then the following properties are satisfied:
		\begin{enumerate}
			\item\label{item:Linfty_flat} For any ball $B(x,2r)\subset \B$, it holds
			\begin{align*}
				\|\g w\|_{L^\infty(B(x,r))} \leq \kappa_{\fl} \left( \fint_{B(x,2r)} |\g w|^p \right)^\frac{1}{p}.
			\end{align*}
			
			\item\label{item:osc_flat} For any $\delta\in(0,\sigma_{\fl}]$ and any ball $B(x,r)\subset \B$, it holds
			\begin{align*}
				\osc_{B(x,\delta r)} (\g w) \leq \kappa_{\fl} \delta^{\alpha_{\fl}} \fint_{B(x,r)} \left| \g w - (\g w)_{B(x,r)}\right|.
			\end{align*}
			
			\item\label{item:2nd_der_flat} It holds $|\g w|^{\frac{p-2}{2}}\g w \in W^{1,2}_{\loc}(\B;\R^{n\times N})$ and $|\g w|^{p-2}\g w \in W^{1,\frac{p}{p-1}}_{\loc}(\B;\R^{n\times N})$ together with the following estimates: for any $B(x,2r)\subset \B$,
			\begin{align*}
				& \int_{B(x,r)} \left| \g\left( |\g w|^{\frac{p-2}{2}} \g w \right) \right|^2 \leq \frac{\kappa_{\fl}}{r^2} \int_{B(x,2r)} |\g w|^p,\\
				& \int_{B(x,r)} \left| \g \left( |\g w|^{p-2} \g w \right) \right|^{\frac{p}{p-1}} \leq \frac{ \kappa_{\fl} }{ r^{\frac{p}{p-1}} } \int_{B(x,2r)} |\g w|^p.
			\end{align*}
		\end{enumerate}
	\end{theorem}

	If $g = g(0)$ is constant on $\B$, then up to a change of variable, we claim that we can assume that $g=\xi$. Indeed, let $\sqrt{g}\in S_n^{++}(\R)$ be the square-root of $g$, \textit{i.e.} $\sqrt{g}$ is the unique positive definite matrix satisfying $\sqrt{g}^2 = g$. Then, the map $w(x) := v\left(\sqrt{g^{-1}}\, x \right)$ satisfies $\lap_{\xi,p}w = 0$. Hence, the regularity of $w$ is provided by \Cref{th:reg_flat}. Thus, $v$ has the same regularity properties as the standard $p$-harmonic maps. We record the associated estimates. 
	\begin{corollary}\label{cor:g_constant}
		Let $g$ be a constant metric on $\B$ and let $\Lambda,\lambda>0$ be such that $\lambda \xi \leq g \leq \Lambda \xi$. There exists constants $\kappa_{\cst}>1$, $\alpha_{\cst}\in (0,1)$ and $\sigma_{\cst}\in(0,\frac{1}{2}]$ depending on $n,N,p,\lambda,\Lambda$ such that the following holds. Consider $v\in W^{1,p}(\B;\R^N)$ be a solution to $\lap_{g,p}v=0$. Then the following properties are satisfied:
		\begin{enumerate}
			\item\label{item:Linfty_cst} For any ball $B\left(x,2r \right)\subset \B$, it holds
			\begin{align*}
				\esssup_{B(x,r)} |\g v|_g \leq \kappa_{\cst} \left( \fint_{B\left( x,2r \right)} |\g v|^p_g(y)\, dy \right)^\frac{1}{p}.
			\end{align*}
			
			\item\label{item:osc_cst} For any $\delta\in(0,\sigma_{\cst}]$ and any ball $B(x,r)\subset \B$, it holds
			\begin{align*}
				\osc_{B(x,\delta r)} (\g v) \leq \kappa_{\cst} \delta^{\alpha_{\cst}} \fint_{B(x,r)} \left| \g v(y) - (\g v)_{B(x,r)}\right|_g\, dy.
			\end{align*}
			
			\item\label{item:2nd_der_cst} It holds $|\g v|^{\frac{p-2}{2}}_g\g v \in W^{1,2}_{\loc}(\B;\R^{n\times N})$ and $|\g v|^{p-2}_g\g v \in W^{1,\frac{p}{p-1}}_{\loc}(\B;\R^{n\times N})$ together with the following estimates: for any $B(x,2r)\subset \B$,
			\begin{align*}
				& \int_{B(x,r)} \left| \g\left( |\g v|^{\frac{p-2}{2}}_g \g v \right) \right|^2_g \leq \frac{\kappa_{\cst}}{r^2} \int_{B(x,2r)} |\g v(y)|^p_g\, dy,\\
				& \int_{B(x,r)} \left| \g \left( |\g v|^{p-2}_g \g v \right) \right|^{\frac{p}{p-1}}_g \leq \frac{ \kappa_{\cst} }{ r^{\frac{p}{p-1}} } \int_{B(x,2r)} |\g v(y)|^p_g\, dy.
			\end{align*}
		\end{enumerate}
	\end{corollary}
	
	\begin{proof}
		\Cref{item:Linfty_cst} and \Cref{item:2nd_der_cst} are a direct consequence of the change of variables $w(z) := v\left( \sqrt{g^{-1}}\, z \right)$. The map $w$ satisfies $\lap_{\xi,p}w=0$. Hence, we can apply \Cref{item:Linfty_flat} and \Cref{item:2nd_der_flat} of \Cref{th:reg_flat} to $w$. \Cref{item:osc_cst} can be obtained by following step by step the proof of \cite[Theorem 3.2]{kuusi2018}.

	\end{proof}

	\subsection{Unconstrained $p$-harmonic maps from curved domain}\label{sec:setting_unconstrained_pharm}
	
	The system \eqref{eq:pharm_system_curved} can be seen as a standard $p$-Laplace system in divergence form. Indeed, consider the vector field
	\begin{align*}
		F := |\g v|^{p-2}_\xi \g^\xi v - |\g v|^{p-2}_g \sqrt{\det g}\, \g^g v.
	\end{align*}
	Then \eqref{eq:pharm_system_curved} can be written as
	\begin{align*}
		\lap_{\xi,p} v = \di_\xi \left( F\right)\ \text{ in }\B.
	\end{align*}
	However, a direct application of the methods of \cite{breit2018} do not apply here. 
	Indeed, even if $F\in L^{\frac{p}{p-1}}(\B;\R^{n\times N})$, the ideas of \cite{breit2018} involves from the beginning the maximal function 
	\begin{align*}
		MF(x) := \sup_{x\in B\subset \B} \fint_{B} \left|F - \fint_B F\right|^\frac{p}{p-1}.
	\end{align*} 
	We can have estimate on this function if and only if we have good estimates on $F$ in some $L^q$ for $q>\frac{p}{p-1}$. Since it is not our case, the system \eqref{eq:pharm_system_curved} falls outside of the setting of \cite{breit2018}. The approach of \cite{kuusi2013} seems more suited to this setting. The goal of this section is to prove the following regularity results.
	
	\begin{theorem}\label{th:reg_pharm_curved_domain}
		Let $\Lambda>\lambda>0$ and $g : \B\to S_n^{++}$ such that
		\begin{align}\label{hyp:g_bounded}
			\lambda\xi \leq g\leq \Lambda\xi.
		\end{align}
		There exists positives constants $\kappa\geq 1$, $\gamma\in(0,1)$, $\eps_1\in(0,1)$ and $\sigma_0\in\left(0,\frac{1}{2}\right]$ depending only on $p,n,N,\lambda,\Lambda$ such that the following holds. Consider a map $v\in W^{1,p}(\B;\R^N)$ such that $\lap_{g,p} v= 0$ on $\B$. Then, $v$ satisfies the following properties:
		\begin{enumerate}
			\item\label{item:Linfty_pharm_curved} Assume that there exists a constant metric $g_0$ such that 
			\begin{align}\label{hyp:g_almost_cst}
				\|g-g_0\|_{L^\infty(\B)} + \left\|\id - g^{-1}g_0 \right\|_{L^\infty(\B)} + \left\|\id - g\, g_0^{-1}\right\|_{L^\infty(\B)} \leq \eps_1.
			\end{align}
			Then, it holds $v \in C^{0,\gamma}_{\loc}(\B)$. More precisely, for any ball $B(x,2r)\subset \B$, it holds
			\begin{align*}
				[v]_{C^{0,\gamma}(B(x,r))} \leq \kappa \left( \frac{1}{r^{n-p}} \int_{B(x,2r)} |\g v|^p_g \, d\vol_g \right)^\frac{1}{p}.
			\end{align*}
			
			\item\label{item:osc_pharm_curved} If $g$ is a $C^{0,\beta}$-metric for some $\beta\in(0,1)$ and $[g]_{C^{0,\beta}(\B)}\leq \eps_1$, then for any $\delta\in(0,\sigma_0]$ and any ball $B(x,r)\subset \B$, it holds:
			\begin{equation*}
				\left\{
				\begin{aligned}
					& \esssup_{B(x,r/2)} |\g v|^p_g  \leq \kappa \fint_{B(x,r)} |\g v|^p_g\, d\vol_g,\\
					& \osc_{B(x,\delta r)} \g v  \leq \kappa \left( 1+ [g]_{C^{0,\beta}(\B)}^\frac{1}{p} \right) \delta^{\gamma} \fint_{B(x,r)} \abs{ \g v - (\g v)_{B(x,r)} }_g\, d\vol_g.
				\end{aligned}
				\right.
			\end{equation*}
			\item\label{item:2nd_der_pharm_curved} If $g$ is a $C^1$-metric, then the second derivatives of $v$ belong to the space $L^2_{\loc}(\B)$ for the measure $|\g v|^{p-2}_g d\vol_g$. In particular, $|\g v|_g^\frac{p-2}{2}\g v \in W^{1,2}_{\loc}(\B;\R^{n\times N})$ and there exists for any ball $B(x,2r)\subset \B$, it holds
			$$
			\int_{B(x,r)} \left| \g^g \left( |\g v|^\frac{p-2}{2}_g \g^g v\right) \right|^2_g\, d\vol_g \leq \frac{\kappa}{r^2}\left( 1+ \|\g g\|_{L^\infty(\B)}^2 \right) \int_{B(x,2r)} |\g v|^p_g\, d\vol_g.
			$$
			As well, it holds $|\g v|^{p-2}_g\g v \in W^{1,\frac{p}{p-1}}_{\loc}(\B;\R^{n\times N})$ and for any ball $B(x,2r)\subset \B$, it holds
			\begin{align*}
				\int_{B(x,r)} \abs{ \g^g (|\g v|^{p-2}_g \g^g v) }_g^\frac{p}{p-1}\, d\vol_g \leq \frac{\kappa}{r^\frac{p}{p-1}}\left( 1+ \|\g g\|_{L^\infty(\B)}^{\frac{p}{p-1}} \right) \int_{B(x,2r)} |\g v|^p_g\, d\vol_g.
			\end{align*}
		\end{enumerate}
		Furthermore, all the constants $\kappa,\gamma,\sigma_0,\eps_1$ are bounded from above as long as $p$ is bounded away from $+\infty$ and as long as the eigenvalues of $g$ are bounded from below.
	\end{theorem}

	\subsection{Almost constant metric}\label{sec:g_Linfty}
	
	In this section, we prove \Cref{item:Linfty_pharm_curved} of \Cref{th:reg_pharm_curved_domain} in the following lemma. The strategy is to compare on each ball $B(x,r)$ the solution to $\lap_{g_0,p}w=0$, whose regularity is provided by \Cref{cor:g_constant}.
	
	\begin{lemma}\label{th:reg_continuous_metric}
		Let $g:\B\to S_n^{++}$ and $g_0\in S_n^{++}$ be such that \eqref{hyp:g_bounded} and \eqref{hyp:g_almost_cst} hold. Let $v$ be a solution to \eqref{eq:pharm_system_curved}. Then, there exists $\kappa\geq 1$ and $\gamma\in(0,1)$ depending only on $n,N,p,\lambda,\Lambda$ such that for any ball $B(x,2r)\subset \B$, it holds
		\begin{align*}
			[v]_{C^{0,\gamma}(B(x,r))} \leq \kappa \left( \frac{1}{r^{n-p+p\gamma}} \int_{B(x,2r)} |\g v|^p_g \, d\vol_g \right)^\frac{1}{p}.
		\end{align*}
	\end{lemma}
	
	\begin{proof}
		Consider a ball $B(x,r)\subset \B$ and define $w\in W^{1,p}(B(x,r);\R^N)$ as the solution to
		\begin{align*}
			\left\{ \begin{array}{l}
				w=v\ \ \ \text{ on }\dr B(x,r),\\
				\lap_{g_0,p} w = 0\ \ \ \text{in }B(x,r).
			\end{array}
			\right.
		\end{align*}
		We consider $v-w$ as test function. It holds
		\begin{align*}
			0 &= \int_{B(x,r)} \left( \scal{ |dv|^{p-2}_g dv }{ d(v-w) }_g\, \sqrt{\det g} - \scal{ |dw|^{p-2}_{g_0} dw }{ d(v-w) }_{g_0}\sqrt{\det g_0}\right).
		\end{align*}
		Hence, we obtain
		\begin{align*}
			& \int_{B(x,r)} \scal{ |dv|^{p-2}_g dv - |dw|^{p-2}_g dw }{ d(v-w) }_g\, d\vol_g \\
			=& \int_{B(x,r)} \left(  \scal{ |dw|^{p-2}_{g_0} dw }{ d(v-w) }_{g_0}\sqrt{\det g_0} - \scal{ |dw|^{p-2}_g dw }{ d(v-w) }_g\, \sqrt{\det g} \right) \\
			=& \int_{B(x,r)} \left( |dw|^{p-2}_{g_0} - |dw|^{p-2}_{g}\right) \scal{  dw }{ d(v-w) }_{g_0}\sqrt{\det g_0} \\
			& + \int_{B(x,r)} |dw|^{p-2}_g\left( \scal{  dw }{ d(v-w) }_{g_0}\sqrt{\det g_0} - \scal{  dw }{ d(v-w) }_{g_0}\, \sqrt{\det g} \right) \\
			& + \int_{B(x,r)} |dw|^{p-2}_g\, \sqrt{\det g}\, \left( \scal{  dw }{ d(v-w) }_{g_0} - \scal{  dw }{ d(v-w) }_{g} \right).
		\end{align*}
		Thanks to \eqref{hyp:g_bounded} and \eqref{hyp:g_almost_cst}, there exists a constant $C=C(n,N,p,\lambda,\Lambda)>0$ such that 
		\begin{align*}
			\int_{B(x,r)} |d(v-w)|^p_g\, d\vol_g \leq C\eps_1 \left( \int_{B(x,r)} |d(v-w)|^p_g\, d\vol_g\right)^{\frac{1}{p}} \left( \int_{B(x,r)} |dw|^p_{g_0}\, d\vol_{g_0}  \right)^{\frac{p-1}{p}}.
		\end{align*}
		Thus, we deduce that 
		\begin{align*}
			\int_{B(x,r)} |d(v-w)|^p_g\, d\vol_g \leq & C\eps_1^{\frac{p}{p-1}} \int_{B(x,r)} |dw|^p_{g_0}\, d\vol_{g_0}.
		\end{align*}
		Since $w$ is minimizing, it holds
		\begin{align*}
			\int_{B(x,r)} |d(v-w)|^p_g\, d\vol_g \leq & C\eps_1^{\frac{p}{p-1}} \int_{B(x,r)} |dv|^p_{g_0}\, d\vol_{g_0}. 
		\end{align*}
		Thanks to \eqref{hyp:g_bounded} and \eqref{hyp:g_almost_cst}, we obtain
		\begin{align}\label{eq:comparison_g_almost_cst}
			\int_{B(x,r)} |d(v-w)|^p_g\, d\vol_g \leq & C\eps_1^{\frac{p}{p-1}} \int_{B(x,r)} |dv|^p_g\, d\vol_g. 
		\end{align}
		On the other hand, for any $\theta\in(0,1)$, it holds
		\begin{align*}
			\int_{B(x,\theta r)} |dv|^p_g d\vol_g &\leq 2^p \int_{B(x,\theta r)} |d(v-w)|^p_g +|dw|^p_g\ d\vol_g \\
			&\leq 2^p\int_{B(x,\theta r)} |d(v-w)|^p_g\, d\vol_g + C\int_{B(x,\theta r)} |dw|^p_{g_0}\, d\vol_{g_0}.
		\end{align*}
		Thanks to \eqref{eq:comparison_g_almost_cst} and \Cref{item:Linfty_cst} of \Cref{cor:g_constant}, the following holds for any $\theta\in(0,1)$:
		\begin{align*}
			\int_{B(x,\theta r)} |dv|^p_g d\vol_g &\leq C \eps_1 \int_{B(x,r)} |dv|^p\, d\vol_g + C\theta^n \int_{B(x,r)} |dw|^p_{g_0}\, d\vol_{g_0} \\
			&\leq C\left( \eps_1 + \theta^n \right) \int_{B(x,r)} |dv|^p\, d\vol_g.
		\end{align*}
		Thus, we obtain
		\begin{align*}
			\frac{1}{(\theta r)^{n-p}}\int_{B(x,\theta r)} |dv|^p_g\, d\vol_g \leq C\left( \frac{\eps_1}{\theta^{n-p}} + \theta^p \right) \frac{1}{r^{n-p}} \int_{B(x,r)} |dv|^p\, d\vol_g.
		\end{align*}
		Choosing $\theta =\frac{1}{4C}$ and then $\eps_1 = \frac{\theta^{n-p}}{4C}$, we deduce the existence of $\alpha\in(0,1)$ depending only on $n,N,p,\lambda,\Lambda$ such that the following holds. For any ball $B(x,r)\subset B(y,s)$, it holds
		\begin{align*}
			\frac{1}{r^{n-p}} \int_{B(x,r)} |dv|^p_g\, d\vol_g \leq \frac{ Cr^\alpha }{ s^{n-p} } \int_{B(y,s)} |dv|^p_g\, d\vol_g.
		\end{align*}
		Thanks to \cite[Theorem 3.1]{han2011}, we obtain the existence of $\gamma\in(0,1)$ and $\kappa>0$ such that
		\begin{align*}
			[v]_{C^{0,\gamma}(B(y,\frac{s}{2}))} \leq \kappa \left(\frac{}{s^{n-p+p\gamma}} \int_{B(y,s)} |dv|^p_g\, d\vol_g\right)^{\frac{1}{p}}.
		\end{align*}
	\end{proof}
	
	\subsection{$C^{0,\beta}$-metric}\label{sec:g_Holder}
	
	In this section, we prove \Cref{item:osc_pharm_curved} of \Cref{th:reg_pharm_curved_domain}. To do so, we follow the strategy of \cite{kuusi2013}. First, we compare on each ball $B(x,r)$, the solution $v$ to the solution of $\lap_{g(x),p} w = 0$ in \Cref{lm:initial_comparison}. We use this estimate in order to obtain a quantitative growth of $\fint_{B(x,r)} |\g v - (\g v)_{B(x,r)}|$ in \Cref{lm:sequence_comparison}. This allows us to obtain a bound on $\|\g v\|_{L^\infty}$ in \Cref{pr:Holder_involves_Lipschitz}. By a bootstrap argument, we obtain a $C^{1,\gamma}$-bound in \Cref{pr:Lipschitz_involves_Holder}.
	
	\begin{theorem}
		Let $\beta \in (0,1)$ and $g$ be a $C^{0,\beta}$-metric on $\bar{\B}$. There exists $\alpha\in(0,1)$ and $C>0$ depending on $n,N,p,g$ such that the following holds. For any solution $v$ to \eqref{eq:pharm_system_curved}, it holds $v\in C^{1,\alpha}_{\loc}(\B)$.
	\end{theorem}
	
	We first compare with a $p$-harmonic extension for a constant metric. 
	
	\begin{lemma}\label{lm:initial_comparison}
		Let $\beta \in (0,1)$ and $g$ be a $C^{0,\beta}$-metric on $\bar{\B}$ satisfying \eqref{hyp:g_bounded}. Consider $v\in W^{1,p}(\B;\R^N)$ a solution to \eqref{eq:pharm_system_curved}. Let $B(x,r)\subset \B$, $g_0 := g(x)$ and $w\in W^{1,p}(B(x,r);\R^N)$ be the solution to 
		\begin{align*}
			\left\{ \begin{array}{l}
				\lap_{g_0,p} w = 0\ \ \text{in }B(x,r),\\
				w = v\ \ \ \ \ \text{on }\dr B(x,r).
			\end{array}
			\right.
		\end{align*}
		Then, there exists $c_1>0$ depending only on $n,N,p,\lambda,\Lambda $ such that the following estimate holds:
		\begin{align*}
			\fint_{B(x,r)} \abs{ d v - d w }^p_g\, d\vol_g \leq c_1[g]_{C^{0,\beta}(\B)} r^\beta \fint_{B(x,r)} |d v|^p_g\, d\vol_g.
		\end{align*}
	\end{lemma}
	
	\begin{proof}
		We proceed exactly as in the proof of \Cref{th:reg_continuous_metric}, replacing \eqref{hyp:g_almost_cst} by the following estimate:
		\begin{align}\label{eq:holder_continuity_g}
			\|g-g(x)\|_{L^\infty(B(x,r))} + \|\id - g^{-1}g(x) \|_{L^\infty(B(x,r))} + \|\id - g\, g(x)^{-1}\|_{L^\infty(B(x,r))} \leq C[g]_{C^{0,\beta}(\B)}r^\beta.
		\end{align}
	\end{proof}
	
	The above \Cref{lm:initial_comparison} shows that $dv$ is very close to $dw$ if $r$ is small. Hence, we can deduce decay estimates on integral quantities of $v$ from those of $p$-harmonic maps with constant metric.
	
	\begin{lemma}\label{lm:sequence_comparison}
		Let $\beta \in (0,1)$ and $g$ be a $C^{0,\beta}$-metric on $\bar{\B}$ satisfying \eqref{hyp:g_bounded}. Consider $v\in W^{1,p}(\B;\R^N)$ a solution to \eqref{eq:pharm_system_curved}. Fix $x\in B(0,\frac{1}{2})$ and $r\in\left( \frac{1}{4}, \frac{1}{2} \right)$ such that $B(x,r)\subset \B$. We define the following quantities:
		\begin{align}
			& \Gamma := \sup_{y_1,y_2\in\B, z\in\s^{n-1} } \frac{ |z|^p_{g(y_1)} \sqrt{\det g(y_1)} }{ |z|^p_{g(y_2)} \sqrt{\det g(y_2)}  },\\
			& \delta := \min\left( \left( \frac{1}{10^{ p} \kappa \Gamma \left(1+\Lambda^{1+\frac{n}{p}} \right) } \right)^\frac{1}{\gamma}, \frac{\sigma_0}{2} \right) \leq \frac{1}{4}, \label{def:choice_delta}\\
			\forall i\in\N, \ \ \ \ & r_i := \delta^i r, \label{def:choice_ri}\\
			\forall i\in\N, \ \ \ \ & B_i := B(x,r_i). \label{def:choice_Bi}
		\end{align}
		Here the constants $\kappa, \sigma_0, \gamma$ are those in \Cref{th:reg_pharm_curved_domain} for the metric $g_0:=g(x)$. \\
		
		There exists $c_2>1$ depending only on $n,N,p,\lambda,\Lambda$ such that for any $i\in\N$, it holds
		\begin{align}\label{eq:decay_sharp_integral}
			\left( \fint_{B_{i+2}} \abs{ dv - (dv)_{B_{i+2}} }^p_\xi \right)^\frac{1}{p} \leq \frac{1}{10} \left( \fint_{B_{i+1}} \abs{ dv - (dv)_{B_{i+1}} }^p_\xi \right)^\frac{1}{p} + c_2[g]_{C^{0,\beta}(\B)}^\frac{1}{p} \delta^\frac{i\beta}{p} \left( \fint_{B_{i+1}} |dv|^p_\xi \right)^\frac{1}{p}.
		\end{align}
		Here, all the averages are taken with respect to the Lebesgue measure. The constant $c_2$ is continuous with respect to  $p$.
	\end{lemma}
	\begin{proof}
		Given $i\in\N$, we consider the $p$-harmonic extension $w_i \in W^{1,p}(B_i;\R^N)$ solution to 
		\begin{equation*}
			\left\{ \begin{array}{l}
				\lap_{g_0,p} w_i = 0\ \ \text{in }B_i,\\
				w_i = v\ \ \ \ \ \text{on }\dr B_i.
			\end{array}
			\right.
		\end{equation*}
		Let $i\in\N$. By triangular inequality, we obtain:
		\begin{align*}
			\left( \fint_{B_{i+2}} \abs{ dv - (dv)_{B_{i+2}} }^p \right)^\frac{1}{p} \leq & \left( \fint_{B_{i+2}} \abs{ dv - dw_{i+1} }^p_\xi \right)^\frac{1}{p} + \left( \fint_{B_{i+2}} \abs{ dw_{i+1} - (dw_{i+1})_{B_{i+2}} }^p_\xi \right)^\frac{1}{p}\\
			& + \abs{ (dv)_{B_{i+2}} - (dw_{i+1})_{B_{i+2}} }_\xi \\
			\leq &\ 2\left( \fint_{B_{i+2}} \abs{ dv - dw_{i+1} }^p_\xi \right)^\frac{1}{p} + \left( \fint_{B_{i+2}} \abs{ dw_i - (dw_{i+1})_{B_{i+2}} }^p_\xi \right)^\frac{1}{p}.
		\end{align*}
		Since $g\geq \lambda \xi$, there exists a constant $c_0>0$ depending on $\lambda$ such that 
		\begin{align*}
			\left( \fint_{B_{i+2}} \abs{ dv - (dv)_{B_{i+2}} }^p \right)^\frac{1}{p} & \leq c_0\left( \fint_{B_{i+2}} \abs{ dv - dw_{i+1} }^p_g\, d\vol_g \right)^\frac{1}{p} + c_0\left( \fint_{B_{i+2}} \abs{ dw_i - (dw_{i+1})_{B_{i+2}} }^p_{g_0}\, d\vol_{g_0} \right)^\frac{1}{p}.
		\end{align*}
		We estimate the last term using \Cref{th:reg_pharm_curved_domain} - \Cref{item:osc_pharm_curved} and \eqref{def:choice_ri}-\eqref{def:choice_Bi}:
		\begin{align*}
			\left( \fint_{B_{i+2}} \abs{ dv - (dv)_{B_{i+2}} }^p_\xi \right)^\frac{1}{p} & \leq c_0\left( \fint_{B_{i+2}} \abs{ dv - dw_{i+1} }^p_g\, d\vol_g \right)^\frac{1}{p} + \left( \kappa \delta^\gamma \fint_{B_{i+1}} \abs{ dw_{i+1} - (dw_{i+1})_{B_{i+1}} }^p_{g_0}\, d\vol_{g_0} \right)^\frac{1}{p}.
		\end{align*}
		Thanks to the choice of $\delta$ in \eqref{def:choice_delta}, we obtain:
		\begin{align*}
			\left( \fint_{B_{i+2}} \abs{ dv - (dv)_{B_{i+2}} }^p_\xi \right)^\frac{1}{p}  \leq &\ c_0\left( \fint_{B_{i+2}} \abs{ dv - dw_{i+1} }^p_g\, d\vol_g \right)^\frac{1}{p} \\
			&+ \frac{1}{10\left(1+\Lambda^{1+\frac{n}{p}} \right)}\left( \fint_{B_{i+1}} \abs{ dw_{i+1} - (dw_{i+1})_{B_{i+1}} }^p_g\, d\vol_g \right)^\frac{1}{p}.
		\end{align*}
		Since $|B_{i+2}| = \delta^n |B_{i+1}|$, we deduce the following:
		\begin{align*}
			\left( \fint_{B_{i+2}} \abs{ dv - (dv)_{B_{i+2}} }^p_\xi \right)^\frac{1}{p} \leq & \frac{c_0}{\delta^{n/p}}\left( \fint_{B_{i+1}} \abs{ dv - dw_{i+1} }^p_g\, d\vol_g \right)^\frac{1}{p} \\
			& + \frac{\Lambda^{1+\frac{n}{p}} }{10 \left(1+\Lambda^{1+\frac{n}{p}} \right)} \left[ \left( \fint_{B_{i+1}} \abs{dv - (dv)_{B_{i+1}} }^p_\xi \right)^\frac{1}{p} \right.\\
			& \left. + \abs{ (dv)_{B_{i+1}} - (dw_{i+1})_{B_{i+1}} }_\xi + \left( \fint_{B_{i+1}} \abs{ dv - dw_{i+1} }^p_\xi \right)^\frac{1}{p} \right] \\
			\leq & \left( \frac{c_0}{\delta^{n/p}} + \frac{1}{5} \right) \left( \fint_{B_{i+1}} \abs{ dv - dw_{i+1} }^p_\xi \right)^\frac{1}{p} + \frac{1}{10} \left( \fint_{B_{i+1}} \abs{ dv - (dv)_{B_{i+1}} }^p_\xi \right)^\frac{1}{p}.
		\end{align*}
		By Lemma \ref{lm:initial_comparison}, we deduce that:
		\begin{align*}
			\left( \fint_{B_{i+2}} \abs{ dv - (dv)_{B_{i+2}} }^p_\xi \right)^\frac{1}{p} & \leq c_2[g]_{C^{0,\beta}(\B)}^\frac{1}{p} \delta^\frac{i\beta}{p} \left( \fint_{B_{i+1}} |dv|^p_\xi \right)^\frac{1}{p} + \frac{1}{10} \left( \fint_{B_{i+1}} \abs{ dv - (dv)_{B_{i+1}} }^p_\xi \right)^\frac{1}{p}. 
		\end{align*}
	\end{proof}
	
	Thanks to the decay in $\delta^i$ of the last term of the right-hand side of \eqref{eq:decay_sharp_integral}, we can sum these estimates and obtain an $L^\infty$-bound on $\g v$. 
	
	\begin{proposition}\label{pr:Holder_involves_Lipschitz}
		Let $\beta \in (0,1)$ and $g$ be a $C^{0,\beta}$-metric on $\bar{\B}$ satisfying \eqref{hyp:g_bounded}. There exists constants $C>0$ and $\eps_1\in(0,1)$ depending on $n,N,p,\lambda,\Lambda$ such that the following holds. Consider $v\in W^{1,p}(\B;\R^N)$ a solution to \eqref{eq:pharm_system_curved}. Assume that $[g]_{C^{0,\beta}(\B)}\leq \eps_1$. Then, it holds $d v \in L^\infty_{\loc} (\B)$ with the following estimate: for any ball $B(x,2r)\subset \B$, we have
		\begin{align*}
			\|dv\|_{L^\infty(B(x,r))} \leq C \left( \frac{1}{r^n} \int_{B(x,2r)} |\g v|^p_\xi \right)^\frac{1}{p}.
		\end{align*}
	\end{proposition}

	\begin{proof} 
		Consider $x\in B(0,\frac{1}{2})$ a Lebesgue point of $|dv|_\xi$, \textit{i.e.} $x$ satisfies 
		\begin{align*}
			\fint_{B(x,r)} |dv(y)|_\xi\, dy \xrightarrow[r\to 0]{}{|dv(x)|_\xi}.
		\end{align*}
		We define the following quantities:
		\begin{align}
			\forall i\in\N,\ \ \ & a_i := \left( \fint_{B_i} \abs{ dv - (dv)_{B_i} }^p_\xi \right)^\frac{1}{p}, \label{def:choice_ai}\\
			& M := \frac{10}{\delta^\frac{n}{p}}\left[ \left( \fint_{B_0} |dv|^p_\xi \right)^\frac{1}{p} + \left( \fint_{B_1} |dv|^p_\xi \right)^\frac{1}{p} \right], \label{def:choice_lambda}\\
			& \ve_1 :=\frac{\delta^n}{10^p} \min\left( \left(  \frac{9}{10 c_2} \frac{ 1-\delta^{\beta/p} }{ \delta^{\beta/p} }\right)^p, 1 \right) \leq \frac{1}{10^p}. \label{def:choice_eps1}
		\end{align}
		Here, the constants $c_2$ and $\delta$ and the balls $B_i$ are defined in Lemma \ref{lm:sequence_comparison}. We prove by induction that for any $i\in\N$, it holds
		\begin{align}\label{eq:induction_hyp}
			\left( \fint_{B_i} |dv|^p_\xi \right)^\frac{1}{p} + a_i \leq M.
		\end{align}
		By definition of $M$ in \eqref{def:choice_lambda}, the above statement is true for $i\in\{0,1\}$ with a slightly better inequality:
		\begin{align}
			\left( \fint_{B_0} |dv|^p_\xi \right)^\frac{1}{p} + a_0 \leq \frac{\delta^\frac{n}{p}}{10} M, \label{eq:start_induction_better0}\\
			\left( \fint_{B_1} |dv|^p_\xi \right)^\frac{1}{p} + a_1 \leq \frac{\delta^\frac{n}{p}}{10} M.\label{eq:start_induction_better1}
		\end{align}
		Let $j\in\N$ and assume that \eqref{eq:induction_hyp} holds for any $i\in\inter{0}{j+1}$. Given $i\in\inter{0}{j+1}$, we apply Lemma \ref{lm:sequence_comparison}:
		\begin{align*}
			a_{i+2} & \leq \frac{a_{i+1}}{10} + c_2[g]_{C^{0,\beta/2}(B_1)}^\frac{1}{p} \delta^\frac{i \beta}{p} \left( \fint_{B_{i+1}} |dv|^p_\xi \right)^\frac{1}{p} \\
			&\leq \frac{a_{i+1}}{10} + c_2 \ve_1^\frac{1}{p} \delta^\frac{i \beta}{p} \left( \fint_{B_{i+1}} |dv|^p_\xi \right)^\frac{1}{p} .
		\end{align*}
		We sum over $i\in\inter{0}{j+1}$, using the assumption \eqref{eq:induction_hyp} to bound the second term in the above inequality:
		\begin{align*}
			\sum_{i=0}^{j+1} a_{i+2} & \leq  \frac{1}{10} \sum_{i=0}^{j+1} a_{i+1} + c_2 \ve_1^\frac{1}{p} M \sum_{i=1}^{j+1} \delta^\frac{i \beta}{p} \\
			&\leq \frac{1}{10} \left( a_1 + \sum_{i=0}^{j+1} a_{i+2} \right) + c_2 \ve_1^\frac{1}{p} M \sum_{i=1}^{j+1} \delta^\frac{i \beta}{p} .
		\end{align*}
		Thus, we obtain
		\begin{align*}
			\sum_{i=0}^{j+1} a_{i+2} & \leq \frac{a_1}{9} +\frac{10 c_2}{9} \ve_1^\frac{1}{p} M \frac{\delta^{\beta/p}}{1-\delta^{\beta/p}}.
		\end{align*}
		By the definition of $\ve_1$ in \eqref{def:choice_eps1}, it holds
		\begin{align*}
			\sum_{i=0}^{j+1} a_{i+2} & \leq \frac{a_1}{9} + \frac{\delta^\frac{n}{p} M }{10}.
		\end{align*}
		By \eqref{eq:start_induction_better1}, we deduce that
		\begin{align}\label{eq:sum_ai}
			\sum_{i=0}^{j+1} a_{i+2} & \leq \frac{\delta^\frac{n}{p} M}{5}.
		\end{align}
		In particular, we obtain the first part of the induction step:
		\begin{align}\label{eq:induction_aj}
			a_{j+2} \leq \frac{M}{2}.
		\end{align}
		To bound $\left( \fint_{B_{j+2}} |dv|^p_\xi \right)^\frac{1}{p}$, we use the triangular inequality:
		\begin{align*}
			\left( \fint_{B_{j+2}} |dv|^p_\xi \right)^\frac{1}{p} &\leq \left( \fint_{B_{j+2}} \abs{ dv - (dv)_{B_{j+1}} }^p_\xi \right)^\frac{1}{p} + \abs{ (dv)_{B_{j+1}} }_\xi \\
			&\leq \left( \fint_{B_{j+2}} \abs{ dv - (dv)_{B_{j+1}} }^p_\xi \right)^\frac{1}{p} + \abs{ (dv)_{B_0} }_\xi + \sum_{i=0}^j \abs{ (dv)_{B_{i+1}} - (dv)_{B_i} }_\xi \\
			&\leq \abs{ (dv)_{B_0} }_\xi + \sum_{i=0}^{j+1} \left( \fint_{B_{i+1}} \abs{ dv - (dv)_{B_i} }^p_\xi \right)^\frac{1}{p} \\
			&\leq \abs{ (dv)_{B_0} }_\xi + \sum_{i=0}^{j+1} \frac{ a_i }{\delta^{n/p}}.
		\end{align*}
		From \eqref{eq:start_induction_better0} and \eqref{eq:sum_ai}, we obtain the second part of the induction step:
		\begin{align*}
			\left( \fint_{B_{j+2}} |dv|^p_\xi \right)^\frac{1}{p} &\leq  \left( \frac{1}{10} + \frac{1}{5} \right) M \leq \frac{M}{2}.
		\end{align*}
		Combining \eqref{eq:induction_aj} and the above inequality shows \eqref{eq:induction_hyp} for $i=j+2$. Thus, for any $j\in\N$, it holds
		\begin{align*}
			\fint_{B_j} |dv|_\xi \leq M.
		\end{align*}
		By constructions of the balls $B_i$, see \eqref{def:choice_Bi}, their common center $x$ is a Lebesgue point for $|dv|_\xi$. Thus,
		\begin{align*}
			|dv(x)|_\xi = \lim_{j\to \infty} \fint_{B_j} |dv|_\xi \leq M.
		\end{align*}
		Since the above inequality is valid for any Lebesgue point in $B_{1/2}$, we deduce that $dv \in L^\infty(B(0,\frac{1}{2}))$. The estimates follows from the above inequality and the choice \eqref{def:choice_lambda}.
	\end{proof}
	
	From the $W^{1,\infty}$-regularity of \Cref{pr:Holder_involves_Lipschitz}, we show $dv \in C^{0,\alpha}$. 
	
	\begin{proposition}\label{pr:Lipschitz_involves_Holder}
		Let $\beta \in (0,1)$ and $g$ be a $C^{0,\beta}$-metric on $\bar{\B}$ satisfying \eqref{hyp:g_bounded}. There exists constants $C>0$, $\eps_1\in(0,1)$ and $\alpha\in(0,1)$ depending on $n,N,p,\lambda,\Lambda$ such that the following holds. Consider $v\in W^{1,\infty}(\B;\R^N)$ a solution to \eqref{eq:pharm_system_curved}. Assume that $[g]_{C^{0,\beta}(\B)} \leq \eps_1$. Then $dv\in C^{0,\alpha}_{\loc}(\B)$ and for any ball $B(x,2r)\subset \B$, it holds
		\begin{align*}
			[dv]_{C^{0,\alpha}(B(x,r))} \leq C\left( \frac{1}{r^{n+\alpha p}}\int_{B(x,2r)} |\g v|^p_\xi \right)^\frac{1}{p}.
		\end{align*}
	\end{proposition}
	
	\begin{proof}[Proof]
		Consider $B(x,r)\subset \B$ and define $g_0 := g(x)$. Let $w\in W^{1,p}(B(x,r);\R^N)$ be the $p$-harmonic extension of $v$ for the metric $g_0$: 
		\begin{align*}
			\left\{ \begin{array}{l}
				\lap_{g_0,p} w = 0\ \ \text{in }B(x,r),\\
				w = v\ \ \ \ \ \text{on }\dr B(x,r).
			\end{array}
			\right.
		\end{align*}
		Let $\sigma_0,\gamma$ be as in \Cref{th:reg_pharm_curved_domain} for the metric $g_0$. Let $\delta\in(0,\sigma_0)$. It holds
		\begin{align*}
			\left( \fint_{B(x,\delta r)} \abs{ dv - (dv)_{B(x,\delta r)} }^p_\xi \right)^\frac{1}{p} &\leq \left( \fint_{B(x,\delta r)} \abs{ dw - (dw)_{B(x,\delta r)} }^p_\xi \right)^\frac{1}{p} + 2\left( \fint_{B(x,\delta r)} \abs{ dv - dw}^p_\xi \right)^\frac{1}{p}.
		\end{align*}
		We estimate the first term thanks to \Cref{th:reg_pharm_curved_domain} - \Cref{item:osc_pharm_curved}:
		\begin{align*}
			\left( \fint_{B(x,\delta r)} \abs{ dv - (dv)_{B(x,\delta r)} }^p_\xi \right)^\frac{1}{p} &\leq \kappa \delta^\gamma \left( \fint_{B(x, r)} \abs{ dw - (dw)_{B(x,r)} }^p_\xi \right)^\frac{1}{p} + 2\left( \fint_{B(x,\delta r)} \abs{ dv - dw}^p_\xi \right)^\frac{1}{p}.
		\end{align*}
		From \Cref{lm:initial_comparison} and the fact that $v$ is Lipschitzian, we estimate the second term: 
		\begin{align*}
			\left( \fint_{B(x,\delta r)} \abs{ dv - (dv)_{B(x,\delta r)} }^p_\xi \right)^\frac{1}{p} &\leq \kappa \delta^\gamma \left( \fint_{B(x, r)} \abs{ dw - (dw)_{B(x,\delta r)} }^p_\xi \right)^\frac{1}{p} + \frac{c_1}{\delta^{n/p}} [g]_{C^{0,\beta}(\B)}^\frac{1}{p} \|dv\|_{L^\infty(B(x,r))} r^\frac{\beta}{p},
		\end{align*}
		where the constant $c_1$ depends on $n,N,p,\lambda,\Lambda$. By triangular inequality, we obtain:
		\begin{align*}
			\left( \fint_{B(x,\delta r)} \abs{ dv - (dv)_{B(x,\delta r)} }^p_\xi \right)^\frac{1}{p} \leq &\ \frac{c_1}{\delta^{n/p}} [g]_{C^{0,\beta}(\B)}^\frac{1}{p} \|dv\|_{L^\infty(B(x,\delta r))} r^\frac{\beta}{p} + \kappa \delta^\gamma \left( \fint_{B(x,r)} \abs{ dv - (dv)_{B(x,\delta r)} }^p_\xi \right)^\frac{1}{p} \\
			& + \kappa \delta^\gamma \left( \fint_{B(x,r)} \abs{ dv - dw }^p_\xi \right)^\frac{1}{p}.
		\end{align*}
		Using again \Cref{lm:initial_comparison} and the fact that $v$ is Lipschitz continuous, we estimate the last term:
		\begin{align*}
			\left( \fint_{B(x,\delta r)} \abs{ dv - (dv)_{B(x,\delta r)} }^p_\xi \right)^\frac{1}{p} &\leq c_1 \left( \delta^{-\frac{n}{p}} +1 \right)[g]_{C^{0,\beta}(\B)}^{\frac{1}{p}} \|dv\|_{L^\infty(B(x,r))} r^\frac{\beta}{p} + \kappa \delta^\gamma \left( \fint_{B(x,r)} \abs{ dv - (dv)_{B(x,r)} }^p_{\xi} \right)^\frac{1}{p}.
		\end{align*}
		Choosing $\delta = \delta(n,N,p,\lambda, \Lambda)$ small enough, we obtain:
		\begin{align*}
			\left( \fint_{B(x,\delta r)} \abs{ dv - (dv)_{B(x,\delta r)} }^p_\xi \right)^\frac{1}{p} &\leq c [g]_{C^{0,\beta}(\B)}^{\frac{1}{p}} \|dv\|_{L^\infty(B(x,r))} r^\frac{\beta}{p} +\frac{1}{2} \left( \fint_{B(x,r)} \abs{ dv - (dv)_{B(x,r)} }^p_\xi \right)^\frac{1}{p},
		\end{align*}
		where the constant $c$ depends only on $n,N,p,\lambda,\Lambda$. Thus, there exists $\alpha\in(0,1)$ such that for all ball $B(x,r)\subset B(y,s)$ and $B(y,2s)\subset \B$,
		\begin{align*}
			\left( \fint_{B(x,r)} \abs{ dv - (dv)_{B(x,r)} }^p_\xi\right)^\frac{1}{p} & \leq c [g]_{C^{0,\beta}(\B)}^\frac{1}{p} \|dv\|_{L^\infty(B(y,s))} r^\alpha \\
			&\leq c [g]_{C^{0,\beta}(\B)}^\frac{1}{p} \left( \frac{1}{s^n}\int_{B(y,2s)} |\g v|^p_\xi \right)^\frac{1}{p} r^\alpha,
		\end{align*}
		where the last inequality follows from \Cref{pr:Holder_involves_Lipschitz}. We conclude the proof thanks to \cite[Theorem 3.1]{han2011}.
	\end{proof}

	\subsection{$C^1$-metric}\label{sec:g_C1}
	
	Assuming that the metric is $C^1$, we can now proceed as in \cite{uhlenbeck1977}: we apply the method of finite differences in order to show the existence of higher derivatives.
	
	\begin{theorem}\label{th:existence_hessian_pharm_curved}
		Let $g$ be a $C^1$-metric satisfying \eqref{hyp:g_bounded}. There exists a constant $C>0$ depending on $n,N,p,\lambda,\Lambda$ satisfying the following.\\
		
		Any solution $v$ to \eqref{eq:pharm_system_curved} has second derivatives of $v$ in the space $L^2_{\loc}(\B)$ for the measure $|\g v|^{p-2}_g d\vol_g$. In particular, $|\g v|_g^\frac{p-2}{2}\g v \in W^{1,2}_{\loc}(\B;\R^{n\times N})$ and for any ball $B(x,2r)\subset \B$, it holds
		\begin{align*}
			\int_{B(x,r)} \left| \g^g \left( |\g v|^\frac{p-2}{2}_g \g^g v\right) \right|^2_g\, d\vol_g \leq C\left( \frac{1}{r^2}+ \|\g g\|_{L^\infty(B(x_0,2r))}^2\right) \int_{B(x,2r)} |\g v|^p_g\, d\vol_g.
		\end{align*}
		As well, it holds $|\g v|^{p-2}_g\g v \in W^{1,\frac{p}{p-1}}_{\loc}(\B;\R^{n\times N})$ and for any ball $B(x,2r)\subset B$, it holds
		\begin{align*}
			\int_{B(x,r)} \abs{ \g^g (|\g v|^{p-2}_g \g^g v) }_g^\frac{p}{p-1}\, d\vol_g \leq C\left( \frac{1}{r^\frac{p}{p-1}} + \|\g g\|_{L^\infty(B(x_0,2r))}^{\frac{p}{p-1}} \right) \int_{B(x,2r)} |\g v|^p_g\, d\vol_g.
		\end{align*}
		Furthermore, $C$ is bounded from above as long as $p$ is bounded away from $+\infty$.
	\end{theorem}
	
	\begin{proof}[Proof]
		Let $B(x_0,r)\subset \B$. Given $i\in\inter{1}{n}$ and $h>0$, we denote $D_{i,h} v : x \mapsto  \frac{1}{h} \left( v(x+he_i) - v(x) \right)$. Let $\eta \in C^\infty_c(B(x_0,\frac{3}{4}r);[0,1])$ such that $\eta = 1$ in $B(x_0,\frac{r}{2})$ and $r|\g \eta|\leq C(n)$. We consider $D_{i,h}^*\left[ \eta^2 D_{i,h} v \right]$ as test function, where $D_{i,h}^*$ is the adjoint operator of $D_{i,h}$. It holds
		\begin{align}
			0 = & \int_{B\left(x_0,\frac{3r}{4} \right)} \scal{\dr_{\alpha} \left[ \eta^2 D_{i,h} v\right]}{ D_{i,h} \left[g^{\alpha\beta} \sqrt{\det g}\, |\g v|^{p-2}_g \dr_{\beta} v\right]}\, dx \nonumber\\
			\geq & \int_{B \left(x_0, \frac{3r}{4} \right)} \eta^2 \scal{D_{i,h} (\dr_{\alpha} v) }{ D_{i,h} \left(g^{\alpha\beta} \sqrt{\det g}\,  |\g v|^{p-2}_g \dr_{\beta} v\right) } dx \nonumber\\
			&  - 2\|\g \eta\|_{L^\infty} \int_{B \left(x_0,\frac{3r}{4}  \right)} |D_{i,h} v|_\xi |\eta|\, \left|D_{i,h} \left( \sqrt{\det g}\, |\g v|^{p-2}_g \g^g v \right) \right|_\xi\, dx. \label{eq:finite_diff}
		\end{align}
		We define $v_t = v + ht D_{i,h} v = (1-t) v + t v(\cdot +he_i)$, for $t\in[0,1]$. Then, it holds
		\begin{align*}
			\scal{D_{i,h} (\dr_{\alpha} v) }{ D_{i,h} \left(g^{\alpha\beta}\sqrt{\det g}\, |\g v|^{p-2}_g \dr_{\beta} v\right) } = & \frac{1}{h} \scal{D_{i,h} (\dr_{\alpha} v) }{ g^{\alpha\beta}_1 \sqrt{\det g_1}|\g v_1|^{p-2}_{g_1} \dr_{\beta} v_1 - g_0^{\alpha\beta} \sqrt{\det g_0}|\g v_0|^{p-2}_{g_0} \dr_{\beta} v_0 }\\
			= & \frac{1}{h}\int_0^1 \scal{ D_{i,h}(\dr_{\alpha} v) }{ \frac{d}{dt} \left( g^{\alpha\beta}_t \sqrt{\det g_t }\, |\g v_t|^{p-2}_{g_t} \dr_{\beta} v_t \right) } dt.
		\end{align*}
		We have 
		\begin{align*}
			\frac{d}{dt}\left( |\g v_{t}|_{g_t}^{p-2} \right) &= \frac{d}{dt} \left[ \left( g^{\alpha\beta}_t \scal{\dr_{\alpha} v_{t} }{ \dr_{\beta} v_{t} } \right)^{\frac{p-2}{2}} \right] \\
			&= \frac{p-2}{2}h \left(2|\g v_{t}|^{p-4}_{g_t} \scal{D_{i,h}(\g v) }{ \g v_{t} }_{g_t} + |\g v_{t}|^{p-4}_{g_t} g_t^{\alpha a}\, D_{i,h} (g_{ab})\, g_t^{b\beta} \scal{\dr_\alpha v_t }{\dr_\beta v_t} \right).
		\end{align*}
		Hence, it holds
		\begin{align*}
			& \scal{D_{i,h} (\dr_{\alpha} v) }{ D_{i,h} \left(g^{\alpha\beta} \sqrt{\det g }\, |\g v|^{p-2}_{g} \dr_{\beta} v\right) }\\
			=& \int_0^1 \scal{ D_{i,h}(\dr_{\alpha} v) }{ g^{\alpha\beta}_t \sqrt{\det g_t }\, |\g v_{t}|^{p-2}_{g_t} D_{i,h} (\dr_{\beta} v) +(p-2) |\g v_{t}|^{p-4}_{g_t} \scal{D_{i,h}(\g v) }{ \g v_{t} }_{g_t} g^{\alpha\beta}_t \sqrt{\det g_t}\, \dr_{\beta}v_{t} }\ dt \\
			& + \int_0^1 \scal{ D_{i,h}(\dr_{\alpha} v) }{ \frac{p-2}{2}|\g v_{t}|^{p-4}_{g_t} g_t^{\alpha a}\, D_{i,h} (g_{ab})\, g_t^{b\beta} \scal{\dr_\alpha v_t }{\dr_\beta v_t} g^{\alpha\beta}_t \sqrt{\det g_t}\, \dr_{\beta}v_{t} }\ dt\\
			& + \int_0^1 \scal{ D_{i,h}(\dr_{\alpha} v) }{ \left( g_t^{\alpha a}\, D_{i,h}(g_{ab})\, g^{b\beta}_t\, \sqrt{\det g_t} + g_t^{\alpha\beta} \tr_{g_t}\left(g_t^{-1}\, D_{i,h}(g_t) \right) \right) |\g v_{t}|^{p-2}_{g_t} \dr_\beta v_t }\ dt\\
			\geq & \int_0^1 (p-2)|\g v_t|^{p-4}_{g_t} \scal{\g v_t}{D_{i,h} (\g v)}_{g_t}^2\sqrt{\det g_t} + |\g v_t|^{p-2}_{g_t} |D_{i,h}(\g v)|^2_{g_t} \sqrt{\det g_t} \ dt \\
			& - C\|\g g\|_{L^\infty(B(x_0,r))}\int_0^1 |D_{i,h}(\g v)|_{g_t} |\g v_t|_{g_t}^{p-1} \sqrt{\det g_t}\ dt,
		\end{align*}
		where the constant $C$ in the last line depends only on $n,N,p,\lambda,\Lambda$. Thus, we deduce that 
		\begin{align}
			& \scal{D_{i,h} (\dr_{\alpha} v) }{ D_{i,h} \left(g^{\alpha\beta} \sqrt{\det g }\, |\g v|^{p-2}_{g} \dr_{\beta} v\right) } \nonumber\\
			\geq & \int_0^1 \left( |\g v_t|^{p-2}_{g} |D_{i,h}(\g v)|^2_{g_t} - C\|\g g\|_{L^\infty(B(x_0,r))} |D_{i,h}(\g v)|_{g_t} |\g v_t|_{g_t}^{p-1}\right) \sqrt{\det g_t}\ dt. \label{eq:main_finite_diff}
		\end{align}
		We now focus on the second term of \eqref{eq:finite_diff}. It holds
		\begin{eqnarray}
			\left| D_{i,h} \left[ \sqrt{\det g}\, |\g v|^{p-2}_g \g^g v \right] \right|_g & = & \left| \int_0^1 \frac{d}{dt}\left( \sqrt{\det g_t}\, |\g v_{t}|^{p-2}_{g_t} \g^{g_t} v_{t} \right) \ dt \right|_g  \nonumber \\
			& \leq & \int_0^1 \left( (p-1)|\g v_t|^{p-2}_{g_\lambda} |D_{i,h}(\g v)|_{g} + C\|\g g\|_{L^\infty(B(x_0,r))} |\g v_t|^{p-2}_{g_t} |\g^{g_t} v_t|_g \right)\, \sqrt{\det g_t}\, dt \nonumber\\
			& \leq & C \int_0^1 \left( |\g v_t|^{p-2}_{g_t} |D_{i,h}(\g v)|_{g_t} + \|\g g\|_{L^\infty(B(x_0,r))} |\g v_t|^{p-1}_{g_t} \right)\, \sqrt{\det g_t}\, dt . \label{eq:lower_finite_diff}
		\end{eqnarray}
		Using \eqref{eq:main_finite_diff}-\eqref{eq:lower_finite_diff} into \eqref{eq:finite_diff}, we obtain
		\begin{align*}
			& \int_{B \left(x_0, \frac{3r}{4} \right)} \int_0^1 \eta^2 |\g v_\lambda|^{p-2}_{g_t}  |D_{i,h} (\g v)|^2_{g_t}\, \sqrt{\det g_t}\ dt\, dx \\
			\leq &\ \frac{C}{r} \int_{B \left(x_0, \frac{3r}{4} \right) } \int_0^1 \eta |D_{i,h} v| \left( \int_0^1 |\g v_t|^{p-2}_{g_t}|D_{i,h} (\g v)|_{g_t}\, \sqrt{\det g_t}\ dt\right) \ dx \\
			& + C\|\g g\|_{L^\infty(B(x_0,r))} \int_{B\left(x_0, \frac{3r}{4} \right) } \int_0^1 \eta \left( |D_{i,h} (\g v)|_{g_t} |\g v_{t}|^{p-1}_{g_t} + |D_{i,h} v| |\g v_t|^{p-1}_{g_t} \right)\, \sqrt{\det g_t}\ dt\, dx.
		\end{align*}
		By Hölder inequality, we deduce the following:
		\begin{align*}
			& \int_{B\left(x_0, \frac{3r}{4} \right)} \int_0^1 \eta^2 |\g v_t|^{p-2}_{g_t}  |D_{i,h} (\g v)|^2_{g_t}\, \sqrt{\det g_t}\ dt\, dx \\
			\leq &\ \frac{C}{r^2} \int_{B\left(x_0, \frac{3r}{4} \right) } \eta^2 |D_{i,h} v|^2 \left( \int_0^1 |\g v_t|^{p-2}_g dt \right)\ d\vol_g \\
			& + C \int_{B\left(x_0, \frac{3r}{4} \right) } \int_0^1 \eta^2 \left( \|\g g\|_{L^\infty(B(x_0,r))}^2 |\g v_{t}|^p_{g_t} + \|\g g\|_{L^\infty(B(x_0,r))} |D_{i,h} v| |\g v_t|^{p-1}_{g_t} \right)\, \sqrt{\det g_t}\ dt\, dx.
		\end{align*}
		By taking the limit $h\to 0$ and summing over $i\in\inter{1}{n}$, we obtain
		\begin{align*}
			\int_{B\left(x_0, \frac{r}{2} \right)} |\g v|^{p-2}_g |\g^2 v|^2_g\, d\vol_g & \leq C\left( \frac{1}{r^2}+\|\g g\|_{L^\infty(B(x_0,r))} + \|\g g\|_{L^\infty(B(x_0,r))}^2 \right) \int_{B(x_0,r)} |\g v|^p_g\, d\vol_g \\
			&\leq C\left( \frac{1}{r^2}+ \|\g g\|_{L^\infty(B(x_0,r))}^2 \right) \int_{B(x_0,r)} |\g v|^p_g\, d\vol_g.
		\end{align*}
		We conclude the proof thanks to two inequalities. On one hand, it holds 
		\begin{align*}
			\left|\g^g \left(|\g v|^\frac{p-2}{2}_g \g^g v\right) \right|_g \leq \frac{p}{2}|\g v|^\frac{p-2}{2}_g |\g^2 v|_g.
		\end{align*}
		On the other hand, by Hölder inequality, it holds
		\begin{align*}
			& \int_{B\left( x_0,\frac{r}{2} \right)} \abs{ \g\left( |\g v|^{p-2}_g\g v \right) }^\frac{p}{p-1}_g\, d\vol_g \\
			\leq &\ p^{\frac{p}{p-1}}\int_{B\left( x_0,\frac{r}{2} \right)} |\g v|^\frac{p(p-2)}{p-1}_g |\g^2 v|^\frac{p}{p-1}_g\, d\vol_g\\
			\leq &\ p^{\frac{p}{p-1}} \int_{B\left(x_0,\frac{r}{2} \right)} \left( |\g v|_g^\frac{p-2}{2} |\g^2 v|_g \right)^\frac{p}{p-1} |\g v|_g^\frac{p(p-2)}{2(p-1)}\, d\vol_g \\
			\leq &\ p^{\frac{p}{p-1}} \left( \int_{B\left( x_0,\frac{r}{2} \right)} |\g v|^{p-2}_g |\g^2 v|^2_g\, d\vol_g \right)^\frac{p}{2(p-1)} \left( \int_{B\left( x_0,\frac{r}{2} \right)} |\g v|^p_g\, d\vol_g \right)^\frac{p-2}{2(p-1)} \\
			\leq &\ C\left( \frac{1}{ r^{\frac{p}{p-1}} } + \|\g g\|_{L^\infty(B(x_0,r))}^{\frac{p}{p-1}} \right)\left(  \int_{B(x_0,r)} |\g v|^p_g\, d\vol_g \right)^\frac{p}{2(p-1)} \left( \int_{B(x_0,r)} |\g v|^p_g\, d\vol_g \right)^\frac{p-2}{2(p-1)} \\
			\leq &\ C\left( \frac{1}{ r^{\frac{p}{p-1}} }+ \|\g g\|_{L^\infty(B(x_0,r))}^{\frac{p}{p-1}} \right) \int_{B(x_0,r)} |\g v|^p_g\, d\vol_g.
		\end{align*}
	\end{proof}

	\subsection{$L^\infty$-bound up to the boundary}\label{sec:Boundary_Linfty}
	
	In order to study the regularity of systems of the form $|\lap_{g,p}u|\aleq |\g u|^p$, we will need $L^\infty$-bounds of solutions to $\lap_{g,p}v=0$ up to the boundary. 
	
	\begin{proposition}\label{pr:boundary_Linfty}
		Consider $\Omega\subset \R^n$ a Lipschitz open set and $g$ an $L^\infty$-metric on $\Omega$ satisfying \eqref{hyp:g_bounded}. Let $\vp \in C^0(\dr \Omega;\R^N)$, $p>2$ and define $v \in W^{1,p}(\Omega;\R^N)$ as the solution to 
		\begin{align}\label{eq:syst_pextention}
			\left\{ \begin{array}{l}
				\lap_{g,p} v = 0 \ \text{ in }\Omega,\\
				v = \vp \ \ \text{ on } \dr \Omega.
			\end{array}
			\right.
		\end{align}
		Then, the image of $v$ lies into the convex envelope of the image of $\vp$. Consequently, it holds
		\begin{align}\label{eq:Linfty}
			\|v\|_{L^\infty(\Omega)} \leq \|\vp \|_{L^\infty(\dr \Omega)}.
		\end{align}
	\end{proposition}
	
	The estimate \eqref{eq:Linfty} can be found in \cite[Lemma A.1]{mazowiecka2020} when $g$ is the flat metric. The authors proved \eqref{eq:Linfty} by an approximation argument. Here, we provide a direct proof combined with a geometric view-point.
	
	\begin{proof}[Proof]
		The convex envelope of the image of $\vp$ is the intersection of all the half-spaces $H$ such that $\mathrm{Im}(\vp)$ lies into $H$. Such half-spaces are characterised by a vector $Z \in\R^N$ and a constant $\alpha\in\R$ such that 
		\begin{align*}
			\forall x\in \dr \Omega,\ \ \ \scal{\vp(x)}{Z} \geq \alpha.
		\end{align*}
		It holds $\di_g\left( |\g v|^{p-2}_g \g^g(\scal{ v}{Z}-\alpha) \right) = 0$ on $\Omega$, with $\scal{v}{Z} -\alpha \geq 0$ on $\dr \Omega$. Thus the negative part $f := \left( \scal{v}{Z}-\alpha \right)_-$ of the function $\scal{v}{Z}-\alpha$ is an admissible test function. It holds:
		\begin{align*}
			\int_{\Omega} |\g v|^{p-2}_g |\g f|^2_g\, d\vol_g = \int_{\Omega} |\g v|^{p-2}_g \g^g (\scal{ v}{Z}-\alpha) \cdot \g^g f\, d\vol_g = 0.
		\end{align*}
		Hence, we obtain $|\g v|^{p-2}_g |\g f|^2_g = 0$ everywhere in $\Omega$. Furthermore, it holds
		\begin{align*}
			|\g f|^p_g \leq |Z|^{p-2}|\g v|^{p-2}_g |\g f|^2_g.
		\end{align*}
		Hence, $f$ is constant in $\Omega$, with $0$ as boundary values. We conclude that $\scal{v}{Z}\geq \alpha$ in $\Omega$.
	\end{proof}
	
	\section{An application: systems of the form $|\lap_{g,p} u|\aleq |\g u|^p$}\label{sec:Critical_system}
	
	Let $g$ be a $C^0$-metric on $\overline{\B}$ satisfying \eqref{hyp:g_bounded}. In this section, we study the regularity of solutions $u\in W^{1,n}\cap C^0(\B;\R^N)$ to systems of the form
	\begin{align}\label{eq:system_complex}
		\lap_{g,p} u = f(u,\g u),
	\end{align}
	where $f : \R^N \times \R^{n\times N}\to \R^N$ is a smooth map satisfying, for some given $\Gamma >0$,
	\begin{align}\label{hyp:LHS}
		\forall (x,X)\in \R^N \times \R^{n\times N},\ \ \ \ |f(x,X)|\leq \Gamma |X|^p_g.
	\end{align}
	The weak formulation is the following:
	\begin{align*}
		\forall \vp \in C^\infty_c(\B;\R^N),\ \ \ \int_{\B} \scal{|\g u|^{p-2}_g \g^g u}{\g^g \vp}_g\, d\vol_g = \int_{\B} f(u,\g u)\cdot \vp\, d\vol_g.
	\end{align*}
	The goal of this section is to prove the following result. 
	
	\begin{theorem}\label{th:reg_nLap_critical}
		Let $p\in[2,+\infty)$, $\beta_1,\beta_2\in (0,1)$ and $g$ be a $L^{\infty}$-metric on $\overline{\B}$ satisfying \eqref{hyp:g_bounded}. There exists $\alpha\in(0,1)$ and $\kappa>0$ depending only on $n,N,\lambda,\Lambda,p,\beta_1,\beta_2$ such that the following holds. Consider any map $u\in W^{1,p}\cap C^0(\B;\R^N)$ solution to \eqref{eq:system_complex}-\eqref{hyp:LHS}. Then $u$ satisfies the following properties:
		\begin{enumerate}
			\item\label{item:Holder_complicated} In the case $p=n$, it holds $u\in C^{0,\alpha}_{\loc}(\B;\R^N)$ together with the following estimate: 
			\begin{align*}
				\forall B(x,2r)\subset \B,\ \ \ [u]_{C^{0,\alpha}(B(x,r))}  \leq \frac{\kappa}{r^{\alpha}}  \|\g u\|_{L^n(B(x,2r))}.
			\end{align*}
			\item\label{item:C1_complicated} In the case $p\neq n$, if $g$ is a $C^{0,\beta_1}$-metric and $u\in C^{0,\beta_2}(\B)$, then $u\in C^{1,\alpha}_{\loc}(\B;\R^N)$. Furthermore, for any ball $B(x,2r)\subset \B$, it holds
			\begin{align*}
				\|\g u\|_{C^{0,\alpha}(B(x,r))} \leq \frac{\kappa}{r^{\frac{n}{p}+\alpha}} \|\g u\|_{L^p(B(x,2r))}.
			\end{align*}
		\end{enumerate}
	\end{theorem}

	We first prove that such a solution $u$ is locally Hölder-continuous (\Cref{item:Holder_complicated}) in \Cref{pr:C0_involves_Holder}. Then, we prove that $u$ is locally Lipschitz continuous in \Cref{pr:Lipschitz_complicated}. Finally, we show that $\g u$ is locally Hölder-continuous (\Cref{item:C1_complicated}) in \Cref{pr:C1_complicated}.\\

	\subsection{Hölder-continuity in the case $p=n$}
	
	First, we start by proving Hölder-continuity in the case $p=n$, which follows from a standard hole-filling trick.
	
	\begin{proposition}\label{pr:C0_involves_Holder}
		Let $g$ be a metric on $\B$ satisfying \eqref{hyp:g_bounded}. There exists $\kappa>0$, $\ve_0\in(0,1)$ and $\alpha\in(0,1)$ depending on $n,N,\Gamma,\lambda,\Lambda$ such that the following holds. Let $\ve\in(0,\ve_0)$. If $u\in W^{1,n}\cap C^0(\B;\R^N)$ satisfies \eqref{eq:system_complex}-\eqref{hyp:LHS} and $\osc_{\B} u \leq \ve$, then $u\in C^{0,\alpha}(B(0,\frac{1}{2}))$, with the following estimate:
		\begin{align*}
			\forall B(x,r)\subset \B,\ \ \ [u]_{C^{0,\alpha}(B(x,\frac{r}{2}))}  \leq \frac{\kappa}{r^{\alpha}}  \|\g u\|_{L^n(B(x,r))}.
		\end{align*}
	\end{proposition}
	
	\begin{proof}
		By assumptions on the oscillations of $u$, for any $x_0\in \B$, it holds :
		\begin{align*}
			\osc_{B(x_0,1-|x_0|)} u \leq \ve_0.
		\end{align*}
		Let $r\in(0,1-|x_0|)$. Consider $\vp \in C^\infty_c(B(x_0,r);[0,1])$ such that $\vp = 1$ in $B(x_0,\frac{r}{2})$ and $r |\g \vp|\leq C(n)$. We use $\vp^n (u-u_0)$, where $u_0 = (u)_{B(x_0,r)\setminus B(x_0,\frac{r}{2})}$, as test function:
		\begin{align*}
			\int_{B(x_0,r)} \scal{ |\g u|^{n-2}_g\g^g u}{ n\vp^{n-1} (\g^g \vp)(u-u_0) + \vp^n \g^g u}_g\, d\vol_g = \int_{B(x_0,r)} f(u,\g u)\cdot \vp^n (u-u_0)\, d\vol_g.
		\end{align*}
		Thus, it holds
		\begin{align*}
			\int_{B(x_0,r)} \vp |\g u|^n_\xi\, dx & \leq C_0\ve \int_{B(x_0,r)} |\g u|^n_\xi\, dx + C_0\int_{B(x_0,r)\setminus B(x_0,\frac{r}{2})} |\g u|^{n-1}_\xi |u-u_0| \vp^{n-1} |\g \vp|_\xi\, dx,
		\end{align*}
		where the constant $C_0$ depends only on $n,N,\Gamma,\lambda,\Lambda$. Hence, it holds
		\begin{align*}
			\int_{B(x_0,r)} \vp |\g u|^n_\xi\, dx & \leq C_0 \ve \int_{B(x_0,r)} |\g u|^n_\xi\, dx + C_0 \left( \int_{B(x_0,r)} |\g u|^n_\xi \vp^n\, dx \right)^\frac{n-1}{n}\left( \int_{B(x_0,r)\setminus B(x_0,\frac{r}{2})} r^{-n}|u-u_0|^n\, dx \right)^\frac{1}{n}.
		\end{align*}
		Thanks to Young's inequality, we obtain
		\begin{align*}
			\int_{B(x_0,r)} \vp |\g u|^n_\xi\, dx \leq C_0 \ve \int_{B(x_0,r)} |\g u|^n_\xi\, dx + C_0\int_{B(x_0,r)\setminus B(x_0,\frac{r}{2})} r^{-n}|u-u_0|^n\, dx.
		\end{align*}
		By Sobolev inequality, we deduce that
		\begin{align*}
			\int_{B(x_0,\frac{r}{2})} |\g u|^n_\xi\, dx \leq \int_{B(x_0,r)} \vp |\g u|^n_\xi\, dx \leq C_0 \ve \int_{B(x_0,r)} |\g u|^n_\xi\, dx + C_0\int_{B(x_0,r)\setminus B(x_0,\frac{r}{2})} |\g u|^n_\xi\, dx.
		\end{align*}
		Thus, we obtain
		\begin{align*}
			\int_{B(x_0,\frac{r}{2})} |\g u|^n_\xi\, dx \leq \frac{ C_0 + C_0 \ve}{1+ C_0} \int_{B(x_0,r)} |\g u|^n_\xi\, dx.
		\end{align*}
		If $\eps = \eps(n,N,\Gamma,\lambda,\Lambda )>0$ is small enough, then it holds $\frac{ C_0 + C_0 \ve}{1+ C_0}\in(0,1)$. The above inequality is valid for any ball $B(x_0,r)$, for $r\in(0,1-|x_0|)$. By iteration, we obtain the existence of constants $\kappa>0$ and $\alpha\in(0,1)$ such that for any ball $B(x_0,r)\subset B(0,\frac{3}{4})$, if $r<1-|x_0|$, it holds
		\begin{align*}
			\int_{B(x_0,r)} |\g u|^n_\xi\, dx \leq \kappa r^{n\alpha} \int_{\B} |\g u|^n_\xi\, dx.
		\end{align*}
		We conclude the proof thanks to \cite[Theorem 3.1]{han2011}.
	\end{proof}
	
	\subsection{$L^\infty$-bound of the gradient}
	
	We now come back to the general case $p\in[2,+\infty)$. We first compare with the $p$-harmonic extension. 
	
	\begin{lemma}\label{lm:comparison_step1}
		Let $g$ be a metric on $\B$ satisfying the pointwise \eqref{hyp:g_bounded}. Let $\alpha\in(0,1)$ and consider $u\in W^{1,p}\cap C^{0,\beta_2}(\overline{\B};\R^N)$ a solution to \eqref{eq:system_complex}-\eqref{hyp:LHS}. Let $B(x,r)\subset \B$ and $v\in W^{1,p}(B(x,r);\R^N)$ be the solution to 
		\begin{align*}
			\left\{ \begin{array}{l}
				\lap_{g,p} v = 0\ \ \text{in }B(x,r),\\
				v = u\ \ \ \ \ \text{on }\dr B(x,r).
			\end{array}
			\right.
		\end{align*}
		Then, there exists $c_1>0$ depending only on $n,N,\Gamma,p$ such that the following estimate holds:
		\begin{align*}
			\fint_{B(x,r)} \abs{ \g u - \g v }^p_g\, d\vol_g \leq c_1 [u]_{C^{0,\beta_2}(B(x,r))} r^{\beta_2} \fint_{B(x,r)} |\g u|^p_g\, d\vol_g.
		\end{align*}
	\end{lemma}
	
	\begin{proof}
		The map $v-(u)_{B(x,r)}$ is also a $p$-harmonic map with boundary value $u-(u)_{B(x,r)}$. Thanks to Proposition \ref{pr:boundary_Linfty}, we obtain
		\begin{align}\label{eq:osc_pextension}
			\|v - (u)_{B(x,r)} \|_{ L^\infty(B(x,r)) } \leq \|u-(u)_{B(x,r)} \|_{L^\infty(B(x,r))} \leq \osc_{B(x,r)} u\leq [u]_{C^{0,\alpha}(\B)} r^\alpha. 
		\end{align}
		We consider $u-v\in W^{1,p}_0\cap L^\infty(B(x,r);\R^N)$ as test function:
		\begin{align*}
			\int_{B(x,r)} \scal{ |\g u|^{p-2}_g \g^g u - |\g v|^{p-2}_g \g^g v }{\g^g(u-v)}_g\, d\vol_g = \int_{B(x,r)} f(u,\g u)\cdot (u-v)\, d\vol_g.
		\end{align*}
		Thanks to \eqref{hyp:LHS}, we obtain a constant $c>0$ depending only on $p,n,N,\Gamma$ such that the following holds:
		\begin{align*}
			\int_{B(x,r)} |\g u - \g v|^p_g\, d\vol_g \leq c \left( \|v - (u)_{B(x,r)} \|_{ L^\infty(B(x,r)) } + \|u - (u)_{B(x,r)} \|_{ L^\infty(B(x,r)) } \right) \int_{B(x,r)} |\g u|^p_g\, d\vol_g.
		\end{align*}
		The result follows from \eqref{eq:osc_pextension}.
	\end{proof}
	
	This lemma shows that $\g u $ is very close to $\g v$ if $r$ is small. We proceed by approximation. \Cref{lm:comparison_step1} provide the following:
	
	\begin{lemma}\label{lm:comparison_step2}
		Let $g$ be a $C^{0,\beta_1}$-metric on $\B$ satisfying \eqref{hyp:g_bounded}. Consider $B(x,r)\subset \B$. We define the following quantities:
		\begin{align}
			& M := \sup_{y_1,y_2\in\B, z\in\s^{n-1} } \frac{ |z|^p_{g(y_1)} \sqrt{\det g(y_1)} }{ |z|^p_{g(y_2)} \sqrt{\det g(y_2)}  },\\
			& \delta := \min\left( \left( \frac{1}{10^{n} M \kappa \left( 1+[g]_{C^{0,\beta_1}(\B)}^\frac{1}{p}\right)\left( 1+[g]_{C^{0,\beta_1}(\B)}^{2p}\right) } \right)^\frac{1}{\gamma}, \frac{\sigma_0}{2} \right) \leq \frac{1}{4}, \label{def:choice_delta2}\\
			\forall i\in\N, \ \ \ \ & r_i := \delta^i r, \label{def:choice_ri2}\\
			\forall i\in\N, \ \ \ \ & B_i := B(x,r_i). \label{def:choice_Bi2}
		\end{align}
		Here the constants $\kappa, \sigma_0, \gamma$ are those in \Cref{th:reg_pharm_curved_domain}. There exists $c_2>1$ depending only on $n,N,\Gamma,\lambda,\Lambda,p,\beta_1,\beta_2$ such that for any $i\in\N$, it holds
		\begin{align*}
			\left( \fint_{B_{i+2}} \abs{ \g u - (\g u)_{B_{i+2}} }^p_\xi\, dy \right)^\frac{1}{p} \leq \frac{1}{10} \left( \fint_{B_{i+1}} \abs{ \g u - (\g u)_{B_{i+1}} }^p_\xi\, dy \right)^\frac{1}{p} + c_2[u]_{C^{0,\beta_2}(\B)}^\frac{1}{p} \delta^\frac{i\beta_2}{p} \left( \fint_{B_{i+1}} |\g u|^p_\xi\, dy \right)^\frac{1}{p}.
		\end{align*}
	\end{lemma}
	\begin{proof}Given $i\in\N$, we consider the $n$-harmonic extension $v_i \in W^{1,p}(B_i;\R^N)$ solution to 
		\begin{align*}
			\left\{ \begin{array}{l}
				\lap_{g,p} v_i = 0\ \ \text{in }B_i,\\
				v_i = u\ \ \ \ \ \text{on }\dr B_i.
			\end{array}
			\right.
		\end{align*}
		Let $i\in\N$. By triangular inequality, we obtain:
		\begin{align*}
			\left( \fint_{B_{i+2}} \abs{ \g u - (\g u)_{B_{i+2}} }^p_\xi\, dy \right)^\frac{1}{p} \leq & \left( \fint_{B_{i+2}} \abs{ \g u - \g v_{i+1} }^p_\xi\, dy \right)^\frac{1}{p} + \left( \fint_{B_{i+2}} \abs{ \g v_{i+1} - (\g v_{i+1})_{B_{i+2}} }^p_\xi\, dy \right)^\frac{1}{p}\\
			& + \abs{ (\g u)_{B_{i+2}} - (\g v_{i+1})_{B_{i+2}} }_\xi \\
			\leq & \ 2\left( \fint_{B_{i+2}} \abs{ \g u - \g v_{i+1} }^p_\xi\, dy \right)^\frac{1}{p} + \left( \fint_{B_{i+2}} \abs{ \g v_i - (\g v_{i+1})_{B_{i+2}} }^p_\xi\, dy \right)^\frac{1}{p}.
		\end{align*}
		We estimate the last term using \Cref{th:reg_pharm_curved_domain} - \Cref{item:osc_pharm_curved}:
		\begin{align*}
			\left( \fint_{B_{i+2}} \abs{ \g u - (\g u)_{B_{i+2}} }^p_\xi\, dy \right)^\frac{1}{p} \leq &\ 2\left( \fint_{B_{i+2}} \abs{ \g u - \g v_{i+1} }^p_\xi\, dy \right)^\frac{1}{p}\\
			& + \left( \kappa\left( 1+[g]_{C^{0,\beta}(\B)}^\frac{1}{p}\right)\|g\|_{L^\infty(\B)}^{2p} \delta^\gamma \fint_{B_{i+1}} \abs{ \g v_{i+1} - (\g v_{i+1})_{B_{i+1}} }^p_\xi\, dy \right)^\frac{1}{p}.
		\end{align*}
		Thanks to the choice of $\delta$ in \eqref{def:choice_delta2}, we obtain:
		\begin{align*}
			\left( \fint_{B_{i+2}} \abs{ \g u - (\g u)_{B_{i+2}} }^n \right)^\frac{1}{p} & \leq 2\left( \fint_{B_{i+2}} \abs{ \g u - \g v_{i+1} }^p_\xi\, dy \right)^\frac{1}{p} + \frac{1}{10}\left( \fint_{B_{i+1}} \abs{ \g v_{i+1} - (\g v_{i+1})_{B_{i+1}} }^p_\xi\, dy \right)^\frac{1}{p}.
		\end{align*}
		Thanks to the relation $|B_{i+2}| = \delta^n |B_{i+1}|$ provided by \eqref{def:choice_ri2}-\eqref{def:choice_Bi2}, we deduce the following:
		\begin{align*}
			\left( \fint_{B_{i+2}} \abs{ \g u - (\g u)_{B_{i+2}} }^p_\xi\, dy \right)^\frac{1}{p} \leq &\ 2\left( \fint_{B_{i+2}} \abs{ \g u - \g v_{i+1} }^p_\xi\, dy \right)^\frac{1}{p} \\
			& + \frac{1}{10} \left[ \left( \fint_{B_{i+1}} \abs{ \g u - (\g u)_{B_{i+1}} }^p_\xi\, dy \right)^\frac{1}{p} \right.\\
			&\left. + \abs{ (\g u)_{B_{i+1}} - (\g v_{i+1})_{B_{i+1}} }_\xi + \left( \fint_{B_{i+1}} \abs{ \g u - \g v_{i+1} }^p_\xi\, dy \right)^\frac{1}{p} \right] \\
			\leq &\ \left( \frac{2}{\delta^{\frac{n}{p}}} + \frac{1}{5} \right) \left( \fint_{B_{i+1}} \abs{ \g u - \g v_{i+1} }^p_\xi\, dy \right)^\frac{1}{p} + \frac{1}{10} \left( \fint_{B_{i+1}} \abs{ \g u - (\g u)_{B_{i+1}} }^p_\xi\, dy \right)^\frac{1}{p}.
		\end{align*}
		By Lemma \ref{lm:comparison_step1}, we deduce that:
		\begin{align*}
			\left( \fint_{B_{i+2}} \abs{ \g u - (\g u)_{B_{i+2}} }^p_\xi\, dy \right)^\frac{1}{p} & \leq c_2 [u]_{C^{0,\beta_2}(B_1)}^\frac{1}{p}\delta^\frac{i\beta_2}{p} \left( \fint_{B_{i+1}} |\g u|^n_\xi\, dy \right)^\frac{1}{n} + \frac{1}{10} \left( \fint_{B_{i+1}} \abs{ \g u - (\g u)_{B_{i+1}} }^p_\xi\, dy \right)^\frac{1}{p}. 
		\end{align*}
	\end{proof}
	
	We now prove the $W^{1,\infty}$-regularity in a similar way as \Cref{pr:Holder_involves_Lipschitz}.
	
	\begin{proposition}\label{pr:Lipschitz_complicated}
		Let $g$ be a $C^{0,\beta_1}$-metric on $\B$ satisfying \eqref{hyp:g_bounded} and let $\alpha\in(0,1)$. There exists $\ve_1\in(0,1)$ and $C>0$ depending only on $n,N,\beta_2,\Gamma,\lambda,\Lambda,\|g\|_{C^{0,\beta_1}(\B)},p$ such that the following holds. If $u\in W^{1,p}\cap C^{0,\beta_2}(\overline{\B};\R^N)$ satisfies \eqref{eq:system_complex}-\eqref{hyp:LHS} and $[u]_{C^{0,\beta_2/2}(\B)} \leq \ve_1$, then $\g u \in L^\infty(B(0,\frac{1}{2}))$ with the estimate 
		\begin{align*}
			\|\g u\|_{L^\infty\left(B\left( 0, \frac{1}{2} \right) \right)} \leq C\|\g u\|_{L^p(\B)}.
		\end{align*}
	\end{proposition}
	
	\begin{remark}
		If $u\in C^{0,\beta_2}(\B)$, then for any ball $B(x,r)\subset \B$, it holds
		\begin{align*}
			[u]_{C^{0,\beta_2/2}(B(x,\rho))} \leq (2 \rho)^\frac{\beta_2}{2}[u]_{C^{0,\beta_2}(\B)}.
		\end{align*}
		Thus, the semi-norm $[u]_{C^{0,\beta_2/2}(\B)}$ can be assumed as small as needed if we restrict the domain.
	\end{remark}
	
	\begin{proof}Consider $x\in B\left( 0, \frac{1}{2}\right)$ a Lebesgue point of $|\g u|_\xi $, \textit{i.e.} the point $x$ satisfies 
		\begin{align*}
			\fint_{B(x,r)} |\g u|_\xi \xrightarrow[r\to 0]{}{|\g u(x)|_\xi}.
		\end{align*}
		Given the balls $B_i$ defined in \eqref{def:choice_Bi2} and $\delta\in(0,1)$ defined in \eqref{def:choice_delta2}, we define the following quantities:
		\begin{align}
			\forall i\in\N,\ \ \ & a_i := \left( \fint_{B_i} \abs{ \g u - (\g u)_{B_i} }^p_\xi\, dy \right)^\frac{1}{p}, \label{def:choice_ai2}\\
			& M := \frac{10}{\delta}\left[ \left( \fint_{B_0} |\g u|^p_\xi\, dy \right)^\frac{1}{p} + \left( \fint_{B_1} |\g u|^p_\xi\, dy \right)^\frac{1}{p} \right], \label{def:choice_lambda2}\\
			& \ve_1 :=\frac{\delta^p}{10^p} \min\left( \left(  \frac{9}{10 c_2} \frac{ 1-\delta^{\beta_2/p} }{ \delta^{\beta_2/p} }\right)^p, 1 \right) \leq \frac{1}{10^p}. \label{def:choice_eps12}
		\end{align}
		Here, the constants $c_2$ and $\delta$ and the balls $B_i$ are defined in Lemma \ref{lm:comparison_step2}. We prove by induction that for any $i\in\N$, it holds
		\begin{align}\label{eq:induction_hyp2}
			\left( \fint_{B_i} |\g u|^p_\xi\, dy \right)^\frac{1}{p} + a_i \leq M.
		\end{align}
		By definition of $M$ in \eqref{def:choice_lambda2}, the above statement is true for $i\in\{0,1\}$ with a slightly better inequality:
		\begin{align}
			\left( \fint_{B_0} |\g u|^p_\xi\, dy \right)^\frac{1}{p} + a_0 \leq \frac{\delta}{10} M, \label{eq:start_induction_better02}\\
			\left( \fint_{B_1} |\g u|^p_\xi\, dy \right)^\frac{1}{p} + a_1 \leq \frac{\delta}{10} M.\label{eq:start_induction_better12}
		\end{align}
		Let $j\in\N$ and assume that \eqref{eq:induction_hyp2} holds for any $i\in\inter{0}{j+1}$. Given $i\in\inter{0}{j+1}$, we apply Lemma \ref{lm:comparison_step2}:
		\begin{align*}
			a_{i+2} & \leq \frac{a_{i+1}}{10} + c_2[u]_{C^{0,\beta_2}(\B)}^\frac{1}{p} \delta^\frac{i \beta_2}{p} \left( \fint_{B_{i+1}} |\g u|^p_\xi\, dy\right)^\frac{1}{p} \\
			&\leq \frac{a_{i+1}}{10} + c_2 \ve_1^\frac{1}{p} \delta^\frac{i\alpha}{p} \left( \fint_{B_{i+1}} |\g u|^p_\xi\, dy\right)^\frac{1}{p} .
		\end{align*}
		We sum over $i\in\inter{0}{j+1}$, using the assumption \eqref{eq:induction_hyp2} to bound the second term in the above inequality:
		\begin{align*}
			\sum_{i=0}^{j+1} a_{i+2} & \leq  \frac{1}{10} \sum_{i=0}^{j+1} a_{i+1} + c_2 \ve_1^\frac{1}{p} M \sum_{i=1}^{j+1} \delta^\frac{i\alpha}{p} \\
			&\leq \frac{1}{10} \left( a_1 + \sum_{i=0}^{j+1} a_{i+2} \right) + c_2 \ve_1^\frac{1}{p} M \sum_{i=1}^{j+1} \delta^\frac{i\alpha}{p} .
		\end{align*}
		Thus, we obtain
		\begin{align*}
			\sum_{i=0}^{j+1} a_{i+2} & \leq \frac{a_1}{9} +\frac{10 c_2}{9} \ve_1^\frac{1}{p} M \frac{\delta^{\alpha/p}}{1-\delta^{\alpha/p}}.
		\end{align*}
		By the definition of $\ve_1$ in \eqref{def:choice_eps12}, it holds
		\begin{align*}
			\sum_{i=0}^{j+1} a_{i+2} & \leq \frac{a_1}{9} + \frac{\delta M }{10}.
		\end{align*}
		By \eqref{eq:start_induction_better12}, we deduce that
		\begin{align}\label{eq:sum_ai2}
			\sum_{i=0}^{j+1} a_{i+2} & \leq \frac{\delta M}{5}.
		\end{align}
		In particular, we obtain the first part of the induction step:
		\begin{align}\label{eq:induction_aj2}
			a_{j+2} \leq \frac{M}{2}.
		\end{align}
		To bound $\left( \fint_{B_{j+2}} |\g u|^p \right)^\frac{1}{p}$, we use the triangular inequality:
		\begin{align*}
			\left( \fint_{B_{j+2}} |\g u|^p_\xi\, dy \right)^\frac{1}{p} &\leq \left( \fint_{B_{j+2}} \abs{ \g u - (\g u)_{B_{j+1}} }^p_\xi\, dy \right)^\frac{1}{p} + \abs{ (\g u)_{B_{j+1}} }_\xi \\
			&\leq \left( \fint_{B_{j+2}} \abs{ \g u - (\g u)_{B_{j+1}} }^p_\xi\, dy \right)^\frac{1}{p} + \abs{ (\g u)_{B_0} }_\xi + \sum_{i=0}^j \abs{ (\g u)_{B_{i+1}} - (\g u)_{B_i} }_\xi \\
			&\leq \abs{ (\g u)_{B_0} }_\xi + \sum_{i=0}^{j+1} \left( \fint_{B_{i+1}} \abs{ \g u - (\g u)_{B_i} }^p_\xi\, dy \right)^\frac{1}{p} \\
			&\leq \abs{ (\g u)_{B_0} }_\xi + \sum_{i=0}^{j+1} \frac{ a_i }{\delta}.
		\end{align*}
		From \eqref{eq:start_induction_better02} and \eqref{eq:sum_ai2}, we obtain the second part of the induction step:
		\begin{align*}
			\left( \fint_{B_{j+2}} |\g u|^p \right)^\frac{1}{p} &\leq  \left( \frac{\delta}{10} + \frac{1}{5} \right) M \leq \frac{M}{2}.
		\end{align*}
		Combining \eqref{eq:induction_aj} and the above inequality shows \eqref{eq:induction_hyp2} for $i=j+2$. Thus, for any $j\in\N$, it holds
		\begin{align*}
			\fint_{B_j} |\g u|_\xi\, dy \leq M.
		\end{align*}
		By constructions of the balls $B_i$, see \eqref{def:choice_Bi2}, their common center $x$ is a Lebesgue point for $|\g u|_\xi$. Thus, we obtain
		\begin{align*}
			|\g u(x)|_\xi = \lim_{j\to \infty} \fint_{B_j} |\g u|_\xi\, dy \leq M.
		\end{align*}
		Since the above inequality is valid for any Lebesgue point in $B(0,\frac{1}{2})$, we deduce that $\g u \in L^\infty(B(0,\frac{1}{2}))$. The estimates follows from the above inequality and the choice \eqref{def:choice_lambda2}.
	\end{proof}
	
	\subsection{$C^{1,\alpha}$-regularity }
	Now, we show $\g u \in C^{0,\alpha}$. 
	
	\begin{proposition}\label{pr:C1_complicated}
		Let $g$ be a $C^{0,\beta_1}$-metric on $\B$ satisfying \eqref{hyp:g_bounded}. There exists $\gamma\in(0,1)$ and $\kappa>0$ depending on $n,N,\Gamma,\lambda,\Lambda,\|g\|_{C^{0,\beta_1}(\B)}$ such that the following holds. If $u\in W^{1,\infty}(\B;\R^N)$ satisfies \eqref{eq:system_complex}-\eqref{hyp:LHS}, then it holds $\g u \in C^{0,\gamma}(B(0,\frac{1}{2}))$ together with the following estimate: 
		\begin{align*}
			[\g u]_{ C^{0,\gamma}\left(B \left(0,\frac{1}{2} \right) \right) } \leq \kappa \left( 1+ [g]_{C^{0,\beta_1}(\B)}^\frac{1}{n} \right)\|\g u \|_{L^\infty(\B)}.
		\end{align*}
	\end{proposition}
	
	\begin{proof}[Proof]
		Consider $B(x,r)\subset \B$ and $v\in W^{1,p}(B(x,r);\R^N)$ the $p$-harmonic extension of $u$: 
		\begin{align*}
			\left\{ \begin{array}{l}
				\lap_{g,p} v = 0\ \ \text{in }B(x,r),\\
				v = u\ \ \ \ \ \text{on }\dr B(x,r).
			\end{array}
			\right.
		\end{align*}
		Let $\sigma_0,\gamma_0$ be as in \Cref{th:reg_pharm_curved_domain}. Let $\delta\in(0,\sigma_0)$. It holds
		\begin{align*}
			\left( \fint_{B(x,\delta r)} \abs{ \g u - (\g u)_{B(x,\delta^2 r)} }^p_\xi\, dy \right)^\frac{1}{p} &\leq \left( \fint_{B(x,\delta r)} \abs{ \g v - (\g v)_{B(x,\delta r)} }^p_\xi\, dy \right)^\frac{1}{p} + 2\left( \fint_{B(x,\delta r)} \abs{ \g u - \g v}^p_\xi\, dy \right)^\frac{1}{p}.
		\end{align*}
		From \Cref{th:reg_pharm_curved_domain} - \Cref{item:osc_pharm_curved}, we estimate the first term:
		\begin{align*}
			\left( \fint_{B(x,\delta r)} \abs{ \g u - (\g u)_{B(x,\delta r)} }^p_\xi\, dy \right)^\frac{1}{p} \leq & \  c_0\left( 1+ [g]_{C^{0,\beta_1}(\B)}^\frac{1}{p} \right) \delta^{\gamma_0} \left( \fint_{B(x,r)} \abs{ \g v - (\g v)_{B(x,r)} }^p_\xi\, dy \right)^\frac{1}{p} \\
			& + 2\left( \fint_{B(x,\delta r)} \abs{ \g u - \g v}^p_\xi\, dy \right)^\frac{1}{p},
		\end{align*}
		where the constant $c_0$ depends on $p,n,N,\lambda,\Lambda,\beta_1$. From \Cref{lm:comparison_step1} and the fact that $u$ is Lipschitz so that we can consider $\beta_2=1$, we estimate the second term: 
		\begin{align*}
			\left( \fint_{B(x,\delta r)} \abs{ \g u - (\g u)_{B(x,\delta r)} }^p_\xi\, dy \right)^\frac{1}{p} &\leq c_0\left( 1+ [g]_{C^{0,\beta_1}(\B)}^\frac{1}{p} \right) \delta^{\gamma_0} \left( \fint_{B(x,r)} \abs{ \g v - (\g v)_{B(x,r)} }^p_\xi\, dy \right)^\frac{1}{p} + c_1\|\g u\|_{L^\infty(\B)}r^\frac{1}{p},
		\end{align*}
		where the constant $c_1$ depends on $n,N,\Gamma,\lambda,\Lambda,p,\beta_1$. By triangular inequality, we obtain:
		\begin{align*}
			\left( \fint_{B(x,\delta r)} \abs{ \g u - (\g u)_{B(x,\delta r)} }^p_\xi\, dy \right)^\frac{1}{p} \leq &\ c_1 \|\g u\|_{L^\infty(\B)} r^\frac{1}{p} + c_0\left( 1+ [g]_{C^{0,\beta_1}(\B)}^\frac{1}{p} \right)\delta^{\gamma_0} \left( \fint_{B(x,r)} \abs{ \g u - (\g u)_{B(x,r)} }^p_\xi\, dy \right)^\frac{1}{p} \\
			& + c_0\left( 1+ [g]_{C^{0,\beta_1}(\B)}^\frac{1}{p} \right)\delta^{\gamma_0} \left( \fint_{B(x,r)} \abs{ \g v - \g u }^p_\xi\, dy \right)^\frac{1}{p}
		\end{align*}
		Using again Lemma \ref{lm:comparison_step1} and the fact that $u$ is Lipschitz continuous, we estimate the last term:
		\begin{align*}
			\left( \fint_{B(x,\delta r)} \abs{ \g u - (\g u)_{B(x,\delta r)} }^p_\xi\, dy \right)^\frac{1}{p} \leq&\  c_1(1+c_0)\left( 1+ [g]_{C^{0,\beta_1}(\B)}^\frac{1}{n} \right)\|\g u \|_{L^\infty(\B)} r^\frac{1}{p} \\
			& + c_0\left( 1+ [g]_{C^{0,\beta_1}(\B)}^\frac{1}{p} \right)\delta^{\gamma_0} \left( \fint_{B(x,r)} \abs{ \g u - (\g u)_{B(x,r)} }^p_\xi\, dy \right)^\frac{1}{p}.
		\end{align*}
		Choosing $\delta = \min\left( (2c_0)^{-\frac{1}{\gamma_0}} \left( 1+ [g]_{C^{0,\beta_1}(\B)}^\frac{1}{p} \right)^{-\frac{1}{\gamma_0}}, \frac{\sigma_0}{2}\right) $, we obtain:
		\begin{align*}
			\left( \fint_{B(x,\delta r)} \abs{ \g u - (\g u)_{B(x,\delta r)} }^p_\xi\, dy \right)^\frac{1}{p} &\leq c_1(1+c_0)\left( 1+ [g]_{C^{0,\beta_1}(\B)}^\frac{1}{p} \right)\|\g u \|_{L^\infty(\B)} r^\frac{1}{p} +\frac{1}{2} \left( \fint_{B(x,r)} \abs{ \g u - (\g u)_{B(x,r)} }^p_\xi\, dy \right)^\frac{1}{p}.
		\end{align*}
		Thus, there exists $\alpha\in(0,1)$ such that 
		\begin{align*}
			\left( \fint_{B(x,r)} \abs{ \g u - (\g u)_{B(x,r)} }^p_\xi\, dy\right)^{\frac{1}{p}} \leq c r^\alpha \left[ \left( 1+ [g]_{C^{0,\beta_1}(\B)}^\frac{1}{p} \right)\|\g u \|_{L^\infty(\B)} + \left( \fint_{\B} \abs{ \g u - (\g u)_{\B} }^p_\xi\, dy \right)^\frac{1}{p} \right],
		\end{align*}
		were $c$ depends only on $p,n,N,\Gamma,\lambda,\Lambda,\beta_1$. We conclude thanks to \cite[Theorem 3.1]{han2011}.
	\end{proof}
	
	\section{Declarations}
	\subsection{Ethical approval}
	Not applicable.
	\subsection{Funding}
	Partial support through ANR BLADE-JC ANR-18-CE40-002 is acknowledged. Partial support through the Swiss National Science Foundation, project SNF 200020\_219429, is acknowledged.
	\subsection{Availability of data and materials}
	Not applicable.
	
	\bibliographystyle{alpha}
	\bibliography{bib.bib}
	
\end{document}